\documentclass[11pt,reqno]{amsart}
\usepackage{graphicx}
\usepackage{color}
\usepackage{amsmath,amssymb,amsthm,amsfonts}
\usepackage{hyperref}
\usepackage{mathdots}
\usepackage{graphicx}
\usepackage{color}
\usepackage{multicol}
\makeatletter

\newcommand{\Rmnum}[1]{\expandafter\@slowromancap\romannumeral #1@}
\makeatother

\newtheorem{thm}{Theorem}[section]

\newtheorem{lemma}[thm]{Lemma}
\newtheorem{remark}[thm]{Remark}
\newtheorem{theorem}[thm]{Theorem}

\usepackage[numbers,sort&compress]{natbib}
\allowdisplaybreaks

\begin{document}

\author{Wei Dai}
\address{Wei Dai
\newline\indent School of Mathematical Sciences \newline\indent Beihang University (BUAA) \newline\indent Beijing 100191, P. R. China \newline\indent and
\newline\indent Key Laboratory of Mathematics, Informatics and Behavioral Semantics \newline\indent Ministry of Education \newline\indent Beijing 100191, P. R. China}
\email{weidai@buaa.edu.cn}

\author{Leyun Wu}
\address{Leyun Wu
\newline\indent Department of Applied Mathematics\newline\indent Hong Kong Polytechnic University\newline\indent Hung Hom, Kowloon, Hong Kong, P. R. China\newline\indent and
\newline\indent School of Mathematical Sciences, MOE-LSC,
\newline\indent Shanghai Jiao Tong University, Shanghai, P. R. China}
\email{leyunwu@126.com}

\title[Uniform boundedness for $n$-th order Lane-Emden system]{Uniform a priori estimates for $n$-th order Lane-Emden system in $\mathbb{R}^{n}$ with $n\geq3$}

\thanks{Wei Dai is supported by the NNSF of China (No. 11971049 and No. 12222102) and the Fundamental Research Funds for the Central Universities, Leyun Wu is supported by the NNSF of China (No. 12031012) and China Postdoctoral Science Foundation (No. 2019M661472).}

\begin{abstract}
In this paper, we establish uniform a priori estimates for positive solutions to the (higher) critical order superlinear Lane-Emden system in bounded domains with Navier boundary conditions in arbitrary dimensions $n\geq3$. First, we prove the monotonicity of solutions for odd order (higher order fractional system) and even order system (integer order system) respectively along the inward normal direction near the boundary by the method of moving planes in local ways. Then we derive uniform a priori estimates by establishing the precise relationships between the maxima of $u$, $v$, $-\Delta u$ and $-\Delta v$ through Harnack inequality. Our results extended the uniform a priori estimates for critical order problems in \cite{KS1,KS2} from $n=2$ to higher dimensions $n\geq3$ and in \cite{CW,DD} from one single equation to system. With such a priori estimates, one will be able to obtain the existence of solutions via topological degree theory or continuation argument.
 \end{abstract}



\maketitle

\numberwithin{equation}{section}

\section{Introduction and main results}

\subsection{Background and setting of the problem}

This paper concerns the superlinear (higher) critical order Lane-Emden system
\begin{equation}\label{eq1-1}
\begin{cases}
(-\Delta)^{\frac{n}{2}} u(x)=v^p(x),&x\in\Omega,\\
(-\Delta)^{\frac{n}{2}} v(x)=u^q(x),&x\in\Omega, \\
u(x)=-\Delta u(x)=\cdots =(-\Delta)^{\frac{n-1}{2}}u(x)=0,\\
v(x)=-\Delta v(x)=\cdots =(-\Delta)^{\frac{n-1}{2}}v(x)=0,
&x\in\mathbb{R}^n \backslash \Omega,\,\, \mbox{if}\,\, n\,\, \mbox{is odd};\\
u(x)=-\Delta u(x)=\cdots =(-\Delta)^{\frac{n}{2}-1}u(x)=0,\\
v(x)=-\Delta v(x)=\cdots =(-\Delta)^{\frac{n}{2}-1}v(x)=0,
&x\in\partial \Omega,\,\, \mbox{if}\,\, n\,\, \mbox{is even};\\
\end{cases}
\end{equation}
in a bounded and strictly convex domain $\Omega \subset \mathbb{R}^n$ with smooth boundary, where $n\geq3$, $p\geq 1$, $q\geq 1$ and $pq>1$.

\medskip

If $n\geq 4$ is even, then $(-\Delta)^{\frac{n}{2}}$ is the (higher) integer order operator, we say that $(u,v)$ is a pair of classical solution to \eqref{eq1-1} provided that $(u,v)\in (C^{n}(\Omega)\cap C^{n-2}(\bar \Omega))^2$ satisfies \eqref{eq1-1}.

\smallskip

If $n\geq 3$ is odd, then $(-\Delta)^{\frac{n}{2}}$ is the higher order fractional operator, which can be defined by
\begin{equation}\label{eq1-2}
(-\Delta)^{\frac{n}{2}}u(x):=(-\Delta)^{\frac{1}{2}}\circ (-\Delta)^{\frac{n-1}{2}} u(x).
\end{equation}
To define the classical solution of higher order fractional system \eqref{eq1-1}, we first introduce the definition of fractional Laplacian. Let $0<\sigma<1$, denote
$$
{\mathcal L}_{2\sigma}:=\left\{u\in L_{loc}^{1}(\mathbb{R}^{n}) \, \bigg| \int_{\mathbb{R}^n }\frac{|u(x)|}{1+|x|^{n+2\sigma}}\mathrm{d}x<+\infty\right\}.
$$
For any $u \in C_{loc}^{[2\sigma], \{2\sigma\}+\varepsilon}(\Omega)\cap   {\mathcal{L}}_{2\sigma}$ with arbitrarily small $\varepsilon>0$, the fractional Laplacian is a nonlocal operator given by
\begin{equation*}
\begin{split}
(-\Delta)^\sigma u(x) &= C_{n, \sigma} \, P.V. \int_{\mathbb{R}^n} \frac{u(x)-u(y)}{|x-y|^{n+2\sigma}} \mathrm{d}y\\
 &= C_{n, \sigma} \lim_{\delta \rightarrow 0^{+}}\int_{\mathbb{R}^{n}\backslash B_{\delta}(x)}
\frac{u(x)-u(y)}{|x-y|^{n+2\sigma}}\mathrm{d}y, \qquad \forall x\in\Omega,
\end{split}
\end{equation*}
where P.V. stands for the Cauchy principal value, $[2\sigma]$ is the integer part of $2\sigma$, and $\{2\sigma\}$ is the fractional part of $2\sigma$. Consequently, in order \eqref{eq1-2} to make sense for any $x\in\Omega$, we require that
 $$
 u \in C_{loc}^{n, \varepsilon}(\Omega) \,\, \mbox{ and } \,  (-\Delta)^\frac{n-1}{2} u \in {\mathcal{L}}_{1}
   $$
with arbitrarily small $\varepsilon >0$. Therefore, when $n$ is odd, we say that $(u, v)$ is a pair of classical solution to \eqref{eq1-1} if and only if $(u, v)$ satisfies \eqref{eq1-1} and $(u,v)\in \left(C_{loc}^{n, \varepsilon}(\Omega)\cap C^{n-1}(\overline{\Omega})\right)^2$.

\smallskip

From the definition of the higher order fractional Laplacian in \eqref{eq1-2}, we can see that it is a nonlocal operator like the fractional Laplacian. This kind of nonlocality makes the nonlocal operators different from the the regular Laplacians, and poses a strong obstacle in the generalization of many fine results from the regular Laplacians to the nonlocal ones. It is well-known that, when considering a Dirichlet problem involving the regular Laplacian in a bounded domain $\Omega$, we only need to prescribe the behavior of solution on $\partial\Omega$, while for the nonlocal operators, we need to impose the condition in the whole $\Omega^{c}:=\mathbb{R}^{n}\setminus\Omega$. The nonlocal operators (e.g. fractional Laplacians) have numerous applications in mathematical physics, probability and finance, such as anomalous diffusion and quasi-geostrophic flows, turbulence and water waves, molecular dynamics, relativistic quantum mechanics of stars, L\'{e}vy process, Brownian motion and Poisson process, etc. For more literature on higher order fractional Laplacian, please c.f. \cite{CDQ,CW,LZ,W} and the references therein.

\medskip

When $n=2$, the system \eqref{eq1-1} becomes the regular second order Lane-Emden system
\begin{equation}\label{eq1-1'}
\begin{cases}
-\Delta u(x)=v^p(x), &x\in\Omega,\\
-\Delta v(x)=u^q(x), &x\in\Omega, \\
u(x)=v(x)=0, &x\in\partial \Omega.\\
\end{cases}
\end{equation}
In \cite{CFM,FMR,HV,M}, the authors derived existence of solutions to \eqref{eq1-1'} for smooth bounded domains $\Omega\subseteq\mathbb{R}^{n}$ provided that $(p,q)$ is below the critical hyperbola (i.e., $\frac{1}{p+1}+\frac{1}{q+1}>\frac{n-2}{n}$), and nonexistence of solutions to \eqref{eq1-1'} for star-shaped domains $\Omega\subseteq\mathbb{R}^{n}$ if $(p,q)$ is equal to or above the critical hyperbola. Thus it is quite interesting to investigate the asymptotic properties of solutions $(u_{p,q},v_{p,q})$ when $(p,q)$ approaches the critical hyperbola from below. For $n\geq3$, \cite{CK,G} shown that, among other things, solutions to \eqref{eq1-1'} must blow up as $(p,q)$ getting close to the critical hyperbola. Recently, Kamburov and Sirakov \cite{KS2} established the uniform boundedness of positive solutions to critical order system \eqref{eq1-1} with $n=2$, that is, solutions are bounded independently of $(p,q)$ provided that $p$ and $q$ are comparable (i.e., $\frac{1}{K}q\leq p\leq Kq$). Being essentially different from the sub-critcial order cases $n\geq3$, the existence or nonexistence of the uniform a priori estimates can not be deduced from a limiting procedure and compactness argument in the critical order case $n=2$, since $(p,q)\rightarrow(+\infty,+\infty)$ as $(p,q)$ tends to the critical hyperbola. With such a priori estimates, one will be able to obtain the existence of solutions via topological degree or continuation arguments (see e.g. \cite{CDQ1,CLM,DQ,F1,FS,MP,QS1,QS2}). For related works on the asymptotic property of solutions as $(p,q)$ tends to the critical hyperbola in the critical order case $n=2$, please see \cite{CLZ} and the references therein. For more literature on existence, non-existence and qualitative properties of solutions to the Lane-Emden system \eqref{eq1-1'}, please refer to \cite{F1,F2,QS2} and the references therein.

\medskip

In particular, when $p=q$, the system \eqref{eq1-1} reduces into one single critical order scalar Lane-Emden equation
\begin{equation}\label{eq1-1''}
\begin{cases}
(-\Delta)^{\frac{n}{2}} u(x)=u^p(x),&x\in\Omega,\\
u(x)=-\Delta u(x)=\cdots =(-\Delta)^{\frac{n-1}{2}}u(x)=0, &x\in\mathbb{R}^n \backslash \Omega,\,\, \mbox{if}\,\, n\,\, \mbox{is odd};\\
u(x)=-\Delta u(x)=\cdots =(-\Delta)^{\frac{n}{2}-1}u(x)=0, &x\in\partial \Omega,\,\, \mbox{if}\,\, n\,\, \mbox{is even}.\\
\end{cases}
\end{equation}
When $n=2$, the uniform boundedness of solutions to the Lane-Emden equation \eqref{eq1-1''} was established by Ren and Wei \cite{RW} for the least-energy solutions and Kamburov and Sirakov \cite{KS1} for general positive solutions. For $n\geq4$ even, Dai and Duyckaerts \cite{DD} derived the uniform a priori estimates for the critical order Lane-Emden equation \eqref{eq1-1''}. Subsequently, for the higher order fractional cases $n\geq3$ odd, Chen and Wu \cite{CW} proved that positive solutions to the Lane-Emden equation \eqref{eq1-1''} are uniformly bounded. For $n\geq2$ even, among other things, \cite{CDQ1,DQ} established a priori estimates for (possibly sign-changing) classical solutions to generalized higher critical order uniformly elliptic equations and existence of positive solutions to critical order Lane-Emden equation \eqref{eq1-1''} for all $p\in(1,+\infty)$. Moreover, the $L^{\infty}$-norm of the positive solutions derived in \cite{CDQ1,DQ} will blow up as $p\rightarrow1^{+}$ provided that $diam \, \Omega<\sqrt{2n}$. For more literature on the asymptotic property of solutions to equation \eqref{eq1-1''} as $p\rightarrow+\infty$, please refer to \cite{H,R} and the references therein. The system case is much more delicate because of the coupling of $u$ and $v$, and the nonlinear interactive effect caused by the two different exponents $p$ and $q$.

\medskip

In this paper, we will investigate the uniform boundedness of positive solutions to the critical order Lane-Emden system \eqref{eq1-1} in general dimensions $n\geq3$, which extend the uniform estimates in \cite{KS1,KS2} from $n=2$ to higher dimensions $n\geq3$ and the uniform estimates in \cite{CW,DD} from scalar Lane-Emden equations to Lane-Emden system.

\subsection{Main results}

Our main result on the uniform boundedness of solutions to the higher critical order Lane-Emden system \eqref{eq1-1} is the following theorem.
\begin{theorem}\label{th1.3}
Let $\Omega \subset \mathbb{R}^n$ ($n\geq 3$) be a bounded and strictly convex domain with smooth boundary and $(u,v)$ be a pair of positive solution to \eqref{eq1-1} with $pq-1\geq\kappa$ for some constant $\kappa>0$. If $n$ is even, we assume $(u,v)\in (C^n(\Omega)\cap C^{n-2}(\bar \Omega))^2$. If $n$ is odd, we assume that $(u, v)\in (C_{loc}^{n, \varepsilon}(\Omega)\cap C^{n-1}(\overline{\Omega})\cap C^{n-2}_{0}(\Omega))^2$. In addition, suppose that
\begin{equation}\label{condition1}
p\leq K q \,\,\,\, (resp. \,\, q\leq K p)
\end{equation}
for some constant $K\geq1$. Then there exists a constant $C>0$ depending only on $n$, $\Omega$, $\kappa$ and $K$ such that any solution $(u,v):=(u_{p,q},v_{p,q})$ of \eqref{eq1-1} satisfies
$$
\|u\|_{L^\infty(\overline{\Omega})} \leq C \,\,\,\, (resp. \,\, \|v\|_{L^\infty(\overline{\Omega})}\leq C).
$$
Furthermore, if $\frac{1}{K}q\leq p\leq K q$, then there exists a constant $C>0$  depending only on $n$, $\Omega$, $\kappa$ and $K$ such that any solution $(u,v)$ of \eqref{eq1-1} are uniformly bounded by $C$ (we emphasize here that $C$ is independent of $p$ and $q$), i.e.,
$$
\|u\|_{L^\infty(\overline{\Omega})} \leq C \quad \mbox{and} \quad \|v\|_{L^\infty(\overline{\Omega})}\leq C.
$$
\end{theorem}

In order to prove Theorem \ref{th1.3}, we need to show the following monotonicity results in boundary layer with uniform thickness w.r.t. $(p,q)$ by using the method of moving planes in local ways.
\begin{theorem}\label{th1.1}
Assume that $n\geq 3$ is odd and $\Omega \subset \mathbb{R}^n$ is a bounded and strictly convex domain with smooth boundary. Suppose that $(u, v)\in \left(C_{loc}^{n, \varepsilon}(\Omega)\cap C^{n-1}(\overline{\Omega})\right)^2$ is a pair of positive classical solution to
\begin{equation*}
\begin{cases}
(-\Delta)^\frac{n}{2} u(x)=v^p(x),&x\in\Omega,\\
(-\Delta)^\frac{n}{2} v(x)=u^q(x),&x\in\Omega, \\
u(x)=-\Delta u(x)=\cdots =(-\Delta)^{\frac{n-1}{2}}u(x)=0, &x\in\mathbb{R}^n\backslash \Omega,\\
v(x)=-\Delta v(x)=\cdots =(-\Delta)^{\frac{n-1}{2}}v(x)=0, &x\in \mathbb{R}^n\backslash \Omega.
\end{cases}
\end{equation*}
Then, for any  $x^0 \in \partial \Omega$, there exists a $\delta_0>0$ depending only on $x^0$ and $\Omega$ such that $(-\Delta)^i u(x)$ and $(-\Delta)^i v(x)$ ($i=0, \cdots, \frac{n-1}{2}$) are strictly monotone increasing along the inward normal direction $\nu^0$ in the region
$$
\tilde{\Sigma}_{\delta_0}=\left\{x\in \bar \Omega \mid 0 \leq (x-x^0)\cdot \nu^0 \leq \delta_0\right\}.
$$
\end{theorem}

\begin{theorem}\label{th1.2}
Assume that $n\geq 4$ is even and $\Omega \subset \mathbb{R}^n$ is a bounded and strictly convex domain with smooth boundary. Let $(u,v)\in (C^n(\Omega)\cap C^{n-2}(\bar \Omega))^2$ be a pair of positive classical solution to
\begin{equation}\label{eq1-4}
\begin{cases}
(-\Delta)^{\frac{n}{2}}u(x)=v^p(x),&x\in\Omega,\\
(-\Delta)^{\frac{n}{2}} v(x)=u^q(x),&x\in\Omega, \\
u(x)=-\Delta u(x)=\cdots =(-\Delta)^{\frac{n}{2}}u(x)=0, &x\in\partial \Omega,\\
v(x)=-\Delta v(x)=\cdots =(-\Delta)^{\frac{n}{2}}v(x)=0, &x\in\partial \Omega.
\end{cases}
\end{equation}
Then, for any $x^0 \in \partial \Omega$, there exists a constant $\delta_0>0$ depending only on $x^0$ and $\Omega$ such that $(-\Delta)^i u(x)$ and $(-\Delta)^i v(x)$ ($i=0,\cdots, {\frac{n}{2}-1}$) are strictly monotone increasing along the inward normal direction $\nu^0$ in the region
$$
\tilde{\Sigma}_{\delta_0}=\left\{x\in \bar \Omega \mid 0 \leq (x-x^0)\cdot \nu^0 \leq \delta_0\right\}.
$$
\end{theorem}

The above monotonicity results implies that the maxima of $u(x)$, $v(x)$, $(-\Delta)^i u(x)$ and $(-\Delta)^i v(x)$ ($i=1,\cdots, {\frac{n}{2}-1}$) are attained in the interior of $\Omega$ away from the boundary, that is, in the set $\{x\in\Omega\mid dist(x,\partial\Omega)\geq\delta\}$. As a consequence, by Theorems \ref{th1.1} and \ref{th1.2}, constructing suitable cut-off functions instead of the first eigenvalues and eigenfunctions, and the integration by parts formula for higher order fractional operators, we can deduce a uniform estimates on $\int_{\Omega}v^{p}\mathrm{d}x$, $\int_{\Omega}u^{q}\mathrm{d}x$ and integrals with logarithmic singularity (see Lemmas \ref{le3.1} and \ref{le4.2}).

\smallskip

Comparing with the planar second order Lane-Emden system \eqref{eq1-1'} in \cite{KS2}, for higher critical order Lane-Emden system \eqref{eq1-1}, we need to derive the monotonicity of $-\Delta u$ and $-\Delta v$ in boundary layer with uniform thickness w.r.t. $(p,q)$, control the maxima of $-\Delta u$ and $-\Delta v$ the speed of $u$ and $v$ when they decrease away from their maxima in terms of maxima of $u$ and $v$ by showing more precise relationships between the maxima of $u$ and $v$, and show that if one of the maxima of $(u,v)$ is large enough then the other is also large. Being different from the single critical order Lane-Emden equation \eqref{eq1-1''} in \cite{CW,DD,KS1}, in order to prove uniform a priori estimates for higher critical order Lane-Emden system \eqref{eq1-1}, it is necessary to establish complicated and precise relationships between the maxima of $u$ and $v$. Consequently, based on the uniform estimates of $\int_{\Omega}v^{p}\mathrm{d}x$, $\int_{\Omega}u^{q}\mathrm{d}x$ and integrals with logarithmic singularity, through combination of a re-scaling argument, the Green representation for (higher order derivatives of) solutions and the inhomogeneous Harnack inequality, we derive uniform a priori estimates by establishing the relationships between the maxima of $u$, $v$, $-\Delta u$ and $-\Delta v$ and estimating how fast the values of $u$ and $v$ can decrease away from their maxima (see Lemmas \ref{le3.2}, \ref{le4.1}, \ref{le4.2}, \ref{le4.3}, \ref{le4.4} and \ref{le4.5}), in which we divided the integral domain into three proper parts and estimate on each part very carefully. In particular, we derived a refined relationship between the maxima of $u$ and $v$ in Lemma \ref{le4.4}, which play a crucial role in controlling the the maxima of $-\Delta u$ and $-\Delta v$ in terms of maxima of $u$ and $v$ and the speed of $u$ and $v$ when they decrease away from their maxima in Lemma \ref{le4.3}. Lemmas \ref{le4.4}, \ref{le4.3} and \ref{le4.5} are the key ingredients in our proof.

\medskip

In Section 2, we use the method of moving planes in local ways to derive the monotonicity of solutions in boundary layer with uniform thickness w.r.t. $(p,q)$ and prove Theorem \ref{th1.1} and \ref{th1.2}. Section 3 is devoted to proving some basic relationships between the maxima of $u$ and $v$. Finally, in Section 4, we establish the uniform a priori estimates and prove Theorem \ref{th1.3}.

\medskip

In what follows, without confusion, we shall use $c$, $C,$ $\gamma_i$ or $C_i$ ($i=0, 1, 2\cdots$) to denote a generic constant which may vary in the context.
In addition, we will denote the solution $(u_{p,q}, v_{p,q})$ by $(u,v)$ for the sake of simplicity.

\section{monotonicity}

In this section, we show that the solutions to \eqref{eq1-1} are strictly monotone increasing along the inward normal direction in a boundary layer of $\Omega$ with uniform thickness. To this end, when $n\geq 3$ is odd, we split the system \eqref{eq1-1} into two fractional equations and $n-1$ integer order equations, then these $n+1$ equations together with the Navier boundary conditions constitute a system. Then applying the direct method of moving planes in local ways, we can obtain the monotonicity results. If $n\geq 4$ is even, we split the system \eqref{eq1-1} into $n$ integer order equations to constitute a system, similarly, using the method of moving planes in local ways, we also can obtain the desired monotonicity results. The proofs for these two cases are similar but not exactly the same since we need to deal with fractional equations and integer order equations in different ways.

\medskip

\subsection{Case 1: $n\geq3$ is odd.}

For any $s=s_0+s_1$, where $s_0\in\mathbb{N}_+$ is the integer part of $s$ and $s_1\in(0,1)$ is the fractional part, the general higher order fractional Laplacian is defined by
$$
(-\Delta)^s u(x):=(-\Delta)^{s_1}\circ(-\Delta)^{s_0} u(x), \quad \forall \, x\in\Omega
$$
for any $u\in C_{loc}^{2s_0+[2s_1], \{2s_1\}+\varepsilon}(\Omega)$ with arbitrary $\varepsilon>0$ such that $(-\Delta)^{s_0}u\in {\mathcal{L}}_{2 s_1}$. Similarly, for any $t=t_0+t_1$, denote $t_0 \in \mathbb{N}_+$ as the integer part of $t$ and $t_1\in(0, 1)$ as the decimal part. If $v\in C_{loc}^{2t_0+[2t_1],\{2t_1\}+\varepsilon}(\Omega)$ with arbitrary $\varepsilon>0$ such that $(-\Delta)^{t_0}u\in {\mathcal{L}}_{2 t_1}$, we can define
$$
(-\Delta)^t v(x):=(-\Delta)^{t_1}\circ(-\Delta)^{t_0} v(x), \quad \forall \, x\in\Omega.
$$

\medskip

We can derive monotonicity results for solutions to more general higher order fractional system (see \eqref{eq2-1}) than \eqref{eq1-1} with $n\geq3$ odd.
\begin{lemma}\label{le2.1}
Let $n\geq1$ and $\Omega \subset \mathbb{R}^n$ be a bounded and strictly convex domain with smooth boundary. Assume that $(u,v)\in \left(C_{loc}^{2s_0+[2s_1], \{2s_1\}+\varepsilon}(\Omega)\cap C^{2s_{0}}(\overline{\Omega})\right)\times\left(C_{loc}^{2t_0+[2t_1], \{2t_1\}+\varepsilon}(\mathbb{R}^n)\cap C^{2t_{0}}(\overline{\Omega})\right)$ is a pair of positive solution to
\begin{equation}\label{eq2-1}
\begin{cases}
(-\Delta)^s u(x)=v^p(x),&x\in\Omega,\\
(-\Delta)^t v(x)=u^q(x),&x\in\Omega, \\
u(x)=-\Delta u(x)=\cdots =(-\Delta)^{s_0}u(x)=0, &x\in\mathbb{R}^n\backslash \Omega,\\
v(x)=-\Delta v(x)=\cdots =(-\Delta)^{t_0}v(x)=0, &x\in\mathbb{R}^n\backslash \Omega,
\end{cases}
\end{equation}
where $s=s_0+s_1, t= t_0+t_1$ with integers $s_0, t_0\in\mathbb{N}_+$ and decimals $s_1, t_1\in(0, 1)$. Then for any  $x^0 \in \partial \Omega,$ there exists a constant  $\delta_0>0$ depending only on $x^0$ and $\Omega$ such that $(-\Delta)^i u(x)$ and $(-\Delta)^j v(x)$ ($i=0,\cdots, s_0$, $j=0,\cdots, t_0$) are strictly monotone increasing along the inward normal direction $\nu^0$ in the region
$$
\tilde{\Sigma}_{\delta_0}=\{x\in \bar \Omega \mid 0 \leq (x-x^0)\cdot \nu^0 \leq \delta_0\}.
$$
\end{lemma}
\begin{proof}
\noindent\textbf{The case $s_0=t_0=1$.} Let $u_1=-\Delta u$. Then \eqref{eq2-1} can be reduced to the system
\begin{equation}\label{eq2-2}
\begin{cases}
-\Delta u(x)=u_1(x),&x\in\Omega,\\
(-\Delta)^{s_1} u_1(x)=v^p(x),&x\in\Omega, \\
-\Delta v(x)=v_1(x),&x\in\Omega,\\
(-\Delta)^{t_1} v_1(x)=u^q(x),&x\in\Omega, \\
u(x)=v(x)=u_1(x)=v_1(x)=0, &x\in \mathbb{R}^n\backslash \Omega,
\end{cases}
\end{equation}
Noting that $u_1$ satisfies
\begin{equation*}
\begin{cases}
(-\Delta)^{s_1} u_1(x)=v^p(x)>0,&x\in\Omega, \\
u_1(x)=0, &x\in\mathbb{R}^n\backslash \Omega,
\end{cases}
\end{equation*}
by the strong maximum principle for the fractional Laplacian (see \cite{CLM}), one has
$$
u_1(x)>0, x\in \Omega
$$
similarly, we have
$$
v_1(x)>0, x\in \Omega.
$$
For any $x^0\in \partial\Omega$, let $\nu^0$ be the unit inward normal vector of $\partial \Omega$ at $x^0$. We will show that there exists a constant $\delta_0>0$ depending only on $x^0$ and $\Omega$ such that $u(x)$  and $v(x)$ are monotone increasing along the inward normal direction in the region
\begin{eqnarray*}
\tilde\Sigma_{\delta_0}=\{x\in \overline{\Omega}\mid 0\leq(x-x^0)\cdot \nu^0\leq \delta_0\}.
\end{eqnarray*}

To this end, we define the moving plane by
$$
T_\lambda=\{x\in \mathbb{R}^n \mid (x-x^0)\cdot \nu^0=\lambda\},
$$
and
$$
\Sigma_\lambda=\{x\in \mathbb{R}^n \mid (x-x^0)\cdot \nu^0<\lambda\}
$$
for $\lambda>0,$ and let $x^\lambda$ be the reflection of the point $x$ about the plane $T_\lambda.$

To compare the values of the solution at two different points, between which one point is obtained from the reflection of the other, we define
$$
u_\lambda(x)=u(x^\lambda),\,\, (u_1)_\lambda(x)=u_1(x^\lambda),\,\, v_\lambda(x)=v(x^\lambda),\,\, (v_1)_\lambda(x)=v_1(x^\lambda),
$$
and
$$U^\lambda(x)=u_\lambda(x)-u(x), \quad  U_1^\lambda(x)=(u_1)_\lambda(x)-u_1(x),$$
$$V^\lambda(x)=v_\lambda(x)-v(x), \quad  V_1^\lambda(x)=(v_1)_\lambda(x)-v_1(x).$$
Then for any $\lambda$ such that the reflection of $\Sigma_\lambda \cap \Omega$ is contained in $\Omega$, we derive from \eqref{eq2-2} that
\begin{equation}\label{eq2-3}
\begin{cases}
-\Delta U^\lambda(x)=U_1^\lambda (x),&x\in\Sigma_\lambda \cap \Omega,\\
(-\Delta)^{s_1} U_1^\lambda(x)=p \xi_\lambda^{p-1}(x)V^\lambda(x),&x\in\Sigma_\lambda \cap \Omega, \\
-\Delta V^\lambda(x)=V_1^\lambda(x),&x\in\Sigma_\lambda \cap \Omega,\\
(-\Delta)^{t_1} V_1^\lambda(x)=q \zeta_\lambda^{q-1}(x)U^\lambda(x),&x\in\Sigma_\lambda \cap \Omega, \\
U^\lambda(x), V^\lambda(x), U_1^\lambda(x), V_1^\lambda(x)\geq 0, &x\in\Sigma_\lambda\backslash(\Sigma_\lambda \cap \Omega),
\end{cases}
\end{equation}
where $\xi_\lambda(x)$ lies in  $v(x)$ and $v_\lambda(x),$ and $\zeta_\lambda(x)$ is valued between $u(x)$ and $u_\lambda(x).$

\medskip

Next we divide the proof into two steps.

\medskip

\textit{Step 1.} In this step, we show that there exists a small constant $\delta>0$ such that
\begin{equation}\label{eq2-4}
U^\lambda(x),\; V^\lambda(x),\; U_1^\lambda(x),\; V_1^\lambda(x)\geq 0, \quad \forall \, x\in \Sigma_\lambda \cap \Omega, \,\, \forall \, 0<\lambda \leq \delta.
\end{equation}

Moreover, we actually have
\begin{eqnarray}\label{eq2-5}
U^\lambda(x),\; V^\lambda(x),\; U_1^\lambda(x),\; V_1^\lambda(x)>0, \quad \forall \, x\in \Sigma _\lambda\cap \Omega, \,\,  \forall\, 0<\lambda\leq \delta.
\end{eqnarray}

We first show \eqref{eq2-4}.

Indeed, \eqref{eq2-4} is a narrow region principle for the system \eqref{eq2-3}. Suppose \eqref{eq2-4} is not valid, we may assume that, for any $\delta>0$ small, there exists a $\lambda\in(0,\delta]$ such that
\begin{eqnarray}\label{eq2-7}
U^\lambda(x)<0 \,\, \mbox{somewhere in}\,\, \Sigma_\lambda \cap \Omega.
\end{eqnarray}
Otherwise, if $U^\lambda(x) \geq 0$ in $\Sigma_\lambda \cap \Omega$ for some $0<\lambda\leq \delta$, from \eqref{eq2-3} and applying the maximum principle for the regular Laplacian and the fractional Laplacian, we obtain
$$
V_1^\lambda(x)\geq 0,\,\, V^\lambda(x)\geq 0\,\,\mbox{and}\,\, U_1^\lambda(x)\geq 0, \quad \forall \, x\in \Sigma _\lambda\cap \Omega
$$
successively.

Let
$$
\phi(x)=\cos\frac{(x-x^0)\cdot \nu^0}{\delta}.
$$
Then it is obvious that
$$\phi(x)\in [\cos 1, 1],\,\, \frac{-\Delta \phi(x)}{\phi(x)}=\frac{1}{\delta^2}, \quad \forall \, x\in \Sigma_\lambda\cap\Omega. $$

Now we introduce the auxiliary function
$$
\overline {U^\lambda}(x)=\frac{ U^\lambda(x)}{\phi(x)}, \quad \forall \, x\in \Sigma_\lambda\cap\Omega.
$$
It follows from \eqref{eq2-7} that there exists a point $x_0\in \Sigma_\lambda \cap \Omega$ such that
$$
\overline {U^\lambda}(x_0)=\min_{\Sigma_\lambda\cap\Omega} \overline {U^\lambda}(x)<0.
$$
By a direct calculation at this  negative minimum point of $\overline {U^\lambda}(x),$ one has
\begin{equation}\label{eq2-8}
\begin{split}
U_1^\lambda(x_0)=&-\Delta U^\lambda(x_0)\\
=&-\Delta\overline{U^\lambda}(x_0)\phi(x_0)-2\nabla\overline{U^\lambda}(x_0)\cdot \nabla \phi(x_0)-\overline{U^\lambda}(x_0)\Delta\phi(x_0)\\
=&-\Delta\overline{U^\lambda}(x_0)\phi(x_0)-\overline{U^\lambda}(x_0)\Delta\phi(x_0)\\
\leq&-\overline{U^\lambda}(x_0)\Delta\phi(x_0)\\
=&\frac{U^\lambda(x_0)}{\delta^2}<0.
\end{split}
\end{equation}
Thus it indicates that there exists a point $x_1\in \Sigma_\lambda\cap \Omega$ such that
\begin{equation*}
U_1^\lambda(x_1)=\min_{\Sigma_\lambda}U_1^\lambda(x)<0.
\end{equation*}
Then at the negative minimum point of $U_1^\lambda(x),$ using \eqref{eq2-8}, the definition of the fractional Laplacian and the second equation in \eqref{eq2-3},  we obtain
\begin{equation}\label{eq2-10}
\begin{split}
p\xi_\lambda^{p-1}(x_1) V^\lambda(x_1)=&(-\Delta)^{s_1}U_1^\lambda(x_1) \\
=&C_{n, s_1}  PV \left(\int_{\Sigma_\lambda}\frac{U_1^\lambda(x_1)-U_1^\lambda(y)}{|x_1-y|^{n+2s_1}}\mathrm{d}y
+\int_{\Sigma^c_\lambda}\frac{U_1^\lambda(x_1)-U_1^\lambda(y)}{|x_1-y|^{n+2s_1}}\mathrm{d}y\right)\\
=&C_{n, s_1}  PV \left(\int_{\Sigma_\lambda}\frac{U_1^\lambda(x_1)-U_1^\lambda(y)}{|x_1-y|^{n+2s_1}}\mathrm{d}y
+\int_{\Sigma_\lambda}\frac{U_1^\lambda(x_1)+U_1^\lambda(y)}{|x_1-y^\lambda|^{n+2s_1}}\mathrm{d}y\right)\\
\leq &2C_{n, s_1}  U_1^\lambda(x_1) \int_{\Sigma_\lambda}\frac{1}{|x_1-y^\lambda|^{n+2s_1}}\mathrm{d}y\\
\leq&\frac{C_1 U_1^\lambda(x_1)}{\delta^{2s_1}}\\
\leq&\frac{C_1 U_1^\lambda(x_0)}{\delta^{2s_1}}\\
\leq&\frac{C_1 U^\lambda(x_0)}{\delta^{2+2s_1}}<0.
\end{split}
\end{equation}
Define
$$
\overline {V^\lambda}(x)=\frac{ V^\lambda(x)}{\phi(x)}, \quad \forall \, x\in \Sigma_\lambda\cap\Omega.
$$
We obtain from \eqref{eq2-10} that $\overline {V^\lambda}(x_1)<0$ and thus there exists a point $y_0\in \Sigma_\lambda \cap \Omega$ such that
\begin{equation}\label{eq2-11}
\overline {V^\lambda}(y_0)=\min_{\Sigma_\lambda\cap\Omega} \overline {V^\lambda}(x) <0.
\end{equation}
By a similar argument as \eqref{eq2-8}, we derive
\begin{equation}\label{eq2-12}
V_1^\lambda(y_0)=-\Delta V^\lambda(y_0)
\leq\frac{V^\lambda(y_0)}{\delta^2}<0.
\end{equation}
It follows that there exists a point $y_1 \in \Sigma_\lambda\cap \Omega$ such that
\begin{equation*}
V_1^\lambda(y_1)=\min_{\Sigma_\lambda}V_1^\lambda(x)<0.
\end{equation*}
Then by a similar argument as \eqref{eq2-10},   applying \eqref{eq2-11}, \eqref{eq2-12} and the fourth equation in \eqref{eq2-3}, we obtain
\begin{equation*}
q \zeta_\lambda^{q-1}(y_1)U^\lambda(y_1)=(-\Delta)^{t_1} V_1^\lambda(y_1)\leq\frac{ C_2 V_1^\lambda(y_1)}{\delta^{2t_1}}\leq\frac{ C_2 V_1^\lambda(y_0)}{\delta^{2t_1}}\leq\frac{ C_2 V^\lambda(y_0)}{\delta^{2+2t_1}}\leq\frac{ C_3 V^\lambda(x_1)}{\delta^{2+2t_1}}<0,
\end{equation*}
which combining with \eqref{eq2-10} yields that
$$
q \|u\|_{L^\infty(\overline{\Omega})}^{q-1}U^\lambda(y_1)\leq \frac{ C_3 V^\lambda(x_1)}{\delta^{2+2t_1}}\leq \frac{C_4 U^\lambda(x_0)}{p\|v\|_{L^\infty(\overline{\Omega})}^{p-1}\delta^{4+2s_1+2t_1}}.
$$
Since
 $$U^\lambda(y_1)=\overline{U^\lambda}(y_1)\phi(y_1)\geq \overline{U^\lambda}(x_0)\phi(y_1)=U^\lambda(x_0)\frac{\phi(y_1)}{\phi(x_0)}\geq C_5 U^\lambda(x_0),$$
 as a consequence,
 $$
 U^\lambda(x_0)\leq \frac{C_6 U^\lambda(x_0)}{pq\|u\|_{L^\infty(\overline{\Omega})}^{q-1}\|v\|_{L^\infty(\overline{\Omega})}^{p-1}\delta^{2s+2t}},
 $$
which is a contradiction if we let $\delta>0$ efficiently small such that
$$
0<\delta<\left(C_{6}^{-1}pq\|u\|_{L^\infty(\overline{\Omega})}^{q-1}\|v\|_{L^\infty(\overline{\Omega})}^{p-1}\right)^{-\frac{1}{2s+2t}}.
$$
Therefore, we conclude that \eqref{eq2-4} is true.

Second, we show \eqref{eq2-5}. Suppose on the contrary that, for some $\lambda\in(0,\delta]$, there exists a point $\xi_0\in \Sigma_\lambda \cap \Omega$ such that $U^\lambda(\xi_0)=0$ or $U_1^\lambda(\xi_0)=0$. If $U^\lambda(\xi_0)=0$, by \eqref{eq2-4}, we know that $\xi_0$ is one of the minimum points of $U^\lambda(x)$ in $\Sigma_{\lambda}$, it follows that
$$
U_1^\lambda(\xi_0)=-\Delta U^\lambda(\xi_0)\leq 0,
$$
hence one also has $U_1^\lambda(\xi_0)=0$. Using \eqref{eq2-4} again, we deduce that $\xi_0$ is also one of the minimum points of $U_1^\lambda(x)$ in $\Sigma_{\lambda}$, i.e.,
$$
U_1^\lambda(\xi_0)=0=\min_{\Sigma_{\lambda}}U_1^\lambda(x),
$$
hence
\begin{equation}\label{eq2-13}
\begin{split}
&(-\Delta)^{s_1} U_1^\lambda(\xi_0)\\
=&C_{n, s_1} PV \int_{\mathbb{R}^n}\frac{- U_1^\lambda(y)}{|\xi_0-y|^{n+2s_1}}\mathrm{d}y\\
=& C_{n, s_1} PV \left(\int_{\Sigma_\lambda}\frac{- U_1^\lambda(y)}{|\xi_0-y|^{n+2s_1}}\mathrm{d}y
+\int_{\Sigma_\lambda}\frac{-U_1^\lambda(y^\lambda)}{|\xi_0-y^\lambda|^{n+2s_1}}\mathrm{d}y\right)\\
=& C_{n, s_1} PV \int_{\Sigma_\lambda}\left(\frac{1}{|\xi_0-y^\lambda|^{n+2s_1}}-\frac{1}{|\xi_0-y|^{n+2s_1}}\right) U_1^\lambda(y)\mathrm{d}y.
\end{split}
\end{equation}
Since
$$
\frac{1}{|\xi_0-y^\lambda|^{n+2s_1}}-\frac{1}{|\xi_0-y|^{n+2s_1}}<0, \quad \forall \, y\in \Sigma_\lambda
$$
and $U_1^\lambda(x)>0$ on $\partial\Omega\cap\Sigma_\lambda$, then \eqref{eq2-13} indicates that
$$
(-\Delta)^{s_1} U_1^\lambda(\xi_0)<0.
$$
From the second equation in \eqref{eq2-3}, we have
$$V^\lambda(\xi_0)<0,$$
which contradicts \eqref{eq2-4}. Therefore, for any $\lambda\in(0,\delta]$, $U^\lambda(x)>0$ and $U_1^\lambda(x)>0$ in $\Sigma_\lambda\cap\Omega$. Similarly, suppose that there exists a point $\xi_1\in \Sigma_\lambda \cap \Omega$ such that $V^\lambda(\xi_{1})=0$ or $V_1^\lambda(\xi_{1})=0$, we can also derive a contradiction. Thus $V^\lambda(x)>0$ and $V_1^\lambda(x)>0$ in $\Sigma_\lambda\cap\Omega$ for any $\lambda\in(0,\delta]$. Therefore, we have proved \eqref{eq2-5}.

\medskip

\textit{Step 2.} In this step, we keep moving the plane continuously along the inward normal direction $x^0$ to its limiting position as long as the inequalities
\begin{equation}\label{eq2-14}
U^\lambda(x),\; V^\lambda(x),\; U_1^\lambda(x),\; V_1^\lambda(x)\geq 0, \quad \forall \, x\in \Sigma_\lambda
\end{equation}
hold. We will prove that the moving plane procedure can be carried on (with the property \eqref{eq2-14}) as long as the reflection of $\Sigma_\lambda\cap \Omega$ is contained in $\Omega$.

To this end, define
$$
 \lambda_0:=\sup\left\{\lambda \mid U^\mu (x), U_1^\mu(x),  V^\mu (x), V_1^\mu (x) \geq 0, \,\,\forall \,x \in \Sigma_\mu, \,\, \forall \,\mu \leq \lambda \right\},
$$
then $\lambda_{0}\geq\delta$ and $U^{\lambda_0}(x),\; V^{\lambda_0}(x),\;U_1^{\lambda_0}(x),\; V_1^{\lambda_0}(x)\geq0$ in $\Sigma_{\lambda_{0}}$. Suppose on the contrary that there is a small constant $\varepsilon>0$ such that the reflection of domain $\Sigma_{\lambda_0+\varepsilon}\cap \Omega$ is still contained in $\Omega$. From \eqref{eq2-5} and its proof in Step 1, we can deduce that
\begin{eqnarray*}
U^{\lambda_0}(x),\; V^{\lambda_0}(x),\;U_1^{\lambda_0}(x),\; V_1^{\lambda_0}(x) >0 \quad \mbox{ in } \, \Sigma_{\lambda_0}\cap \Omega.
\end{eqnarray*}

It follows that for any $\sigma\in(0,\lambda_{0})$ small, there exists a positive constant $c_\sigma$ such that
\begin{eqnarray*}
U^{\lambda_0}(x),\; V^{\lambda_0}(x),\;U_1^{\lambda_0}(x),\; V_1^{\lambda_0}(x) \geq c_\sigma>0 \quad \mbox{ in } \, \overline{\Sigma_{\lambda_0-\sigma}\cap \Omega}.
\end{eqnarray*}
Then the continuity of $U^\lambda,\,\, V^\lambda,\, U^\lambda_1$ and $V^\lambda_1$ with respect to $\lambda$ yields that, there exists a $0<\eta<\min\{\varepsilon,\sigma\}$ small such that
\begin{eqnarray*}
U^{\lambda}(x),\; V^{\lambda}(x),\;U_1^{\lambda}(x),\; V_1^{\lambda}(x) \geq 0, \quad \forall \, x\in \overline{\Sigma_{\lambda_0-\sigma}\cap \Omega}, \,\,\, \forall \, \lambda\in(\lambda_0, \lambda_0+\eta].
\end{eqnarray*}
For any $\lambda\in(\lambda_0, \lambda_0+\eta]$, let
$$
D_{\lambda}:=\left(\Sigma_{\lambda} \setminus \overline{\Sigma_{\lambda_0-\sigma}}\right)\cap \Omega.
$$
If $\sigma$ is sufficiently small, $D_{\lambda}$ is a narrow region. Through entirely similar proof as Step 1, one can obtain that
$$U^\lambda (x),\, V^\lambda (x),\, U_1^\lambda (x),\, V_1^\lambda (x) \geq 0, \quad \forall \, x \in D. $$
It follows immediately that
$$U^\lambda (x),\, V^\lambda (x),\, U_1^\lambda (x),\, V_1^\lambda (x) \geq 0, \quad \forall \, x \in \Sigma_\lambda\cap\Omega, \,\,\, \forall \, \lambda\in(\lambda_0, \lambda_0+\eta],$$
which indicates that the plane $T_\lambda$ can still be moved inward a little bit from $T_{\lambda_0}$ with \eqref{eq2-14} and contradicts the definition of $\lambda_0$.

\medskip

In summary, there exists a constant $\delta_0>0$ depending only on $x^0$ and $\Omega$ such that $u(x), v(x), -\Delta u(x)$ and $-\Delta v(x)$
are strictly monotone increasing along the inward normal direction $\nu^0$ in the region
$$
\tilde \Sigma_{\delta_0}=\left\{x\in\overline{\Omega} \mid 0\leq (x-x^0)\cdot \nu^0\leq \delta_0\right\}.
$$

\medskip

\noindent\textbf{The case $s_0t_0>1$.} In contrast with the cases $s_0>1$ and $t_0>1$, the cases $s_0=1$ and $t_0>1$ \emph{or} $t_0=1$ and $s_0>1$ are simple. Therefore, without loss of generality, we may assume that $s_0>1$ and $t_0>1$.

For $i=1,\cdots, s_0$ and $j=1,\cdots, t_0$, let
$$u_i(x):=(-\Delta)^i u(x)  \quad \mbox{and} \quad v_j(x)=(-\Delta)^j v(x) $$
and
$$
U_i^\lambda(x)=u_i(x^\lambda)-u_i(x)  \quad \mbox{and} \quad V_j^\lambda(x)=v_j(x^\lambda)-v_j(x).
$$
For any $\lambda$ such that the reflection of $\Sigma_\lambda \cap \Omega$ is contained in $\Omega$,  we derive from \eqref{eq2-1} that
\begin{equation}\label{eq2-17}
\begin{cases}
(-\Delta)^{s_1} U_{s_0}^\lambda(x)=p \xi_\lambda^{p-1}(x)V^\lambda(x),&x\in\Sigma_\lambda \cap \Omega, \\
-\Delta U_{s_0-1}^\lambda(x)=U_{s_0}^\lambda(x) ,&x\in\Sigma_\lambda \cap \Omega,\\
\cdots\cdots\\
-\Delta U^\lambda(x)=U_1^\lambda (x),&x\in\Sigma_\lambda \cap \Omega,\\
(-\Delta)^{t_1} V_{t_0}^\lambda(x)=q \zeta_\lambda^{q-1}(x)U^\lambda(x),&x\in\Sigma_\lambda \cap \Omega, \\
-\Delta V_{t_0-1}^\lambda(x)=V_{t_0}^\lambda(x) ,&x\in\Sigma_\lambda \cap \Omega,\\
\cdots\cdots\\
-\Delta V^\lambda(x)=V_1^\lambda (x),&x\in\Sigma_\lambda \cap \Omega,\\
U^\lambda(x), V^\lambda(x), U_i^\lambda(x), V_j^\lambda(x)\geq 0, \,\, i=1,\cdots,s_0,\,j=1,\cdots, t_0,  &x\in\Sigma_\lambda\backslash(\Sigma_\lambda \cap \Omega),
\end{cases}
\end{equation}
where $\xi_\lambda(x)$ lies in $v(x)$ and $v_\lambda(x),$ and $\zeta_\lambda(x)$ is valued between $u(x)$ and $u_\lambda(x).$

Analogously, we construct the auxiliary functions:
$$
\overline{U^\lambda}(x)=\frac{U^\lambda(x)}{\phi(x)}, \quad \overline{V^\lambda}(x)=\frac{V^\lambda(x)}{\phi(x)}
$$
and
$$
\overline{U_i^\lambda}(x)=\frac{U_i^\lambda(x)}{\phi(x)}, \quad \overline{V_j^\lambda}(x)=\frac{V_j^\lambda(x)}{\phi(x)}
$$
for $i=1,\cdots, s_0-1, j=1,\cdots t_0-1 $ and any $x \in \Sigma_\lambda\cap \Omega.$

If $U^\lambda(x) <0$ somewhere in $\Sigma_\lambda \cap \Omega$, by a similar argument as in the case $s_0=t_0=1$, we derive from \eqref{eq2-17} that there exists a point $x_0\in \Sigma_\lambda\cap\Omega$ such that
$$
\overline{U^\lambda}(x_0)=\min_{\Sigma_\lambda\cap\Omega}\overline{U^\lambda}(x)<0,
$$
and
$$
U_1^\lambda(x_0)\leq \frac{U^\lambda(x_0)}{\delta^2}<0,
$$
which indicates that there exists a point $x_1\in \Sigma_\lambda\cap\Omega$ such that
$$
\overline{U_1^\lambda}(x_1)=\min_{\Sigma_\lambda\cap\Omega}\overline{U_1^\lambda}(x)<0.
$$
Then at the minimum point $x_1$, one has
$$
U_2^\lambda(x_1)\leq \frac{U_1^\lambda(x_1)}{\delta^2}<0.
$$
Repeating the process in this way, we conclude that there exist $x_i\in \Sigma_\lambda\cap\Omega$ ($i=1, \cdots, s_0-1$) such that
$$
\overline{U_{i}^\lambda}(x_i)=\min_{\Sigma_\lambda\cap\Omega}\overline{U_{i}^\lambda}(x)<0, \quad i=1, \cdots, s_0-1,
$$
and
$$
U_{i}^\lambda(x_{i-1})\leq \frac{U_{i-1}^\lambda(x_{i-1})}{\delta^2}<0, \quad i=2, \cdots, s_0.
$$
Ultimately, we obtain that there exist a point $x_{s_0}\in \Sigma_\lambda\cap\Omega$ such that
$$
{U_{s_0}^\lambda}(x_{s_0})=\min_{\Sigma_\lambda}{U_{s_0}^\lambda}(x)<0,
$$
and
\begin{equation}\label{eq2-18}
\begin{split}
p\xi_\lambda^{p-1}(x_{s_0}) V^\lambda(x_{s_0})=&(-\Delta)^{s_1}U_{s_0}^\lambda(x_{s_0})
\leq\frac{C_7 U_{s_0}^\lambda(x_{s_0})}{\delta^{2s_1}}
\leq\frac{C_7 U_{s_0}^\lambda(x_{s_0-1})}{\delta^{2s_1}}
\leq\frac{C_7 U_{s_0-1}^\lambda(x_{s_0-1})}{\delta^{2+2s_1}}\\
\leq&\frac{C_8 U_{s_0-1}^\lambda(x_{s_0-2})}{\delta^{2+2s_1}}
\leq\frac{C_8 U_{s_0-2}^\lambda(x_{s_0-2})}{\delta^{4+2s_1}}
\leq \cdots
\leq\frac{C_9 U^\lambda(x_0)}{\delta^{2s}}<0,
\end{split}
\end{equation}
which implies that there exists a point $y_0\in \Sigma_\lambda \cap \Omega$ such that
\begin{equation*}
\overline {V^\lambda}(y_0)=\min_{\Sigma_\lambda\cap\Omega} \overline {V^\lambda}(x) <0.
\end{equation*}
Then arguing similarly as in deriving \eqref{eq2-18}, there exist a sequence of $\{y_j\} \subset \Sigma_\lambda\cap \Omega$ ($j=1, \cdots, t_0$) such that
$$
\overline {V_j^\lambda}(y_j)=\min_{\Sigma_\lambda\cap\Omega} \overline {V_j^\lambda}(x)<0, \quad j=1, \cdots, t_0,
$$
and
$$
V_{1}^\lambda(y_{0})\leq \frac{V^\lambda(y_{0})}{\delta^2}<0, \quad V_{j}^\lambda(y_{j-1})\leq \frac{V_{j-1}^\lambda(y_{j-1})}{\delta^2}<0, \quad j=2, \cdots, t_0,
$$
and
\begin{equation*}
\begin{split}
q \zeta_\lambda^{q-1}(y_{t_0})U^\lambda(y_{t_0})=&(-\Delta)^{t_1}V_{t_0}^\lambda(y_{t_0})
\leq\frac{C_{10} V_{t_0}^\lambda(y_{t_0})}{\delta^{2t_1}}
\leq\frac{C_{10} V_{t_0}^\lambda(y_{t_0-1})}{\delta^{2t_1}}
\leq\frac{C_{10} V_{t_0-1}^\lambda(y_{t_0-1})}{\delta^{2+2t_1}}\\
\leq&\frac{C_{11} V_{t_0-1}^\lambda(y_{t_0-2})}{\delta^{2+2t_1}}
\leq\frac{C_{11} V_{t_0-2}^\lambda(y_{t_0-2})}{\delta^{4+2t_1}}
\leq \cdots
\leq\frac{C_{12} V^\lambda(y_0)}{\delta^{2t}}\leq\frac{C_{13} V^\lambda(x_{s_0})}{\delta^{2t}}<0,
\end{split}
\end{equation*}
which together with \eqref{eq2-18} yields
 $$
 U^\lambda(x_0)\leq \frac{C_{14}U^\lambda(x_0)}{pq\|u\|_{L^\infty(\overline{\Omega})}^{q-1}\|v\|_{L^\infty(\overline{\Omega})}^{p-1}\delta^{2s+2t}},
 $$
which is a contradiction if we let $\delta>0$ efficiently small such that
$$
0<\delta<\left(C_{14}^{-1}pq\|u\|_{L^\infty(\overline{\Omega})}^{q-1}\|v\|_{L^\infty(\overline{\Omega})}^{p-1}\right)^{-\frac{1}{2s+2t}}.
$$
Therefore, there exists a $\delta>0$ small enough such that for any $\lambda\in(0,\delta]$,
\begin{eqnarray*}
U^\lambda(x),\; V^\lambda(x),\; U_i^\lambda(x) \; V_j^\lambda(x) \geq 0, \,\,\, i=1,\cdots , s_0, \;j=1,\cdots, t_0, \quad \forall \, x\in \Sigma_\lambda.
\end{eqnarray*}

Finally, through a similar argument as the case $s_0=t_0=1$, we obtain that there exists a constant $\delta_0>0$ depending only on $x^0$ and $\Omega$ such that $(-\Delta)^i u(x)$ ($i=0, \cdots, s_0$) and $(-\Delta)^j v(x)$ ($j=0,\cdots , t_0$) are strictly monotone increasing along the inward normal direction in the region
$$
\tilde\Sigma_{\delta_0}=\left\{x\in \overline{\Omega}\mid 0\leq(x-x^0)\cdot \nu^0\leq \delta_0\right\}.
$$
This completes the proof of Lemma \ref{le2.1}.
\end{proof}

\begin{proof}
[Proof of Theorem \ref{th1.1}]
Theorem \ref{th1.1} is a direct consequence of Lemma \ref{le2.1} by letting $s=t=\frac{n}{2}$ with $n\geq3$ odd and thus concludes the proof.
\end{proof}

\subsection{Case 2: $n\geq4$ is even.}

If $n\geq 4$ is even, then the system \eqref{eq1-4} consists of two (higher) integer order equations, we can also obtain monotonicity results. Consequently, we shall give the proof of Theorem \ref{th1.2}.

We can derive monotonicity results for solutions to more general (higher) integer order system (see \eqref{eq2-1a}) than \eqref{eq1-1} with $n\geq4$ even.
\begin{lemma}\label{le2.1a}
Let $n\geq1$, $s$, $t\geq1$ be integers and $\Omega \subset \mathbb{R}^n$ be a bounded and strictly convex domain with smooth boundary. Assume that $(u,v)\in \left(C^{2s}(\Omega)\cap C^{2s-2}(\overline{\Omega})\right)\times\left(C^{2t}(\Omega)\cap C^{2t-2}(\overline{\Omega})\right)$ is a pair of positive solution to
\begin{equation}\label{eq2-1a}
\begin{cases}
(-\Delta)^s u(x)=v^p(x),&x\in\Omega,\\
(-\Delta)^t v(x)=u^q(x),&x\in\Omega, \\
u(x)=-\Delta u(x)=\cdots =(-\Delta)^{s-1}u(x)=0, &x\in\partial\Omega,\\
v(x)=-\Delta v(x)=\cdots =(-\Delta)^{t-1}v(x)=0, &x\in\partial\Omega.
\end{cases}
\end{equation}
Then for any  $x^0 \in \partial \Omega,$ there exists a constant  $\delta_0>0$ depending only on $x^0$ and $\Omega$ such that $(-\Delta)^i u(x)$ and $(-\Delta)^j v(x)$ ($i=0,\cdots, s-1$, $j=0,\cdots, t-1$) are strictly monotone increasing along the inward normal direction $\nu^0$ in the region
$$
\tilde{\Sigma}_{\delta_0}=\{x\in \bar \Omega \mid 0 \leq (x-x^0)\cdot \nu^0 \leq \delta_0\}.
$$
\end{lemma}
\begin{proof}
The proof of Lemma \ref{le2.1a} is similar to Lemma \ref{le2.1}, the only difference is that when we deal with $U^\lambda_{s_0}(x)$ and $V^\lambda_{t_0}(x)$ with integers $s_{0}:=s-1\geq0$ and $t_{0}:=t-1\geq0$, it is also necessary to introduce the auxiliary functions
$$
\overline{U_{s_0}^\lambda}(x)=\frac{U_{s_0}^\lambda(x)}{\phi(x)}\,\,\,\mbox{and}\,\,\, \overline{V_{t_0}^\lambda}(x)=\frac{V_{t_0}^\lambda(x)}{\phi(x)},
$$
where $U_{0}^\lambda:=U^\lambda$ and $V_{0}^\lambda:=V^\lambda$. In this way, we are able to obtain the estimates
$$
-\Delta U_{s_0}^\lambda(x_{s_0})
\leq\frac{C U_{s_0}^\lambda(x_{s_0})}{\delta^2}
\,\,\,\mbox{and}\,\,\,-\Delta V_{t_0}^\lambda(x_{t_0})
\leq\frac{C V_{t_0}^\lambda(x_{t_0})}{\delta^2},
$$
and the rest of the proof is entirely similar to Lemma \ref{le2.1}. This finishes the proof of Lemma \ref{le2.1a}.
\end{proof}

\begin{proof}
[Proof of Theorem \ref{th1.2}]
Theorem \ref{th1.2} is a direct consequence of Lemma \ref{le2.1a} by letting $s=t=\frac{n}{2}$ with $n\geq4$ even and thus concludes the proof.
\end{proof}

\section{Relationships between $\|u\|_{L^\infty}$ and $\|v\|_{L^\infty}$}

In this section, we mainly focus on the preliminary relationships between $\|u\|_{L^\infty}$ and $\|v\|_{L^\infty}$. More complicated and precise relationships will be given in Section 4. By constructing different suitable  super-solutions for odd $n$  and even $n$  respectively and using the maximum principles for the higher order equations, we derive that $\|u\|_{L^\infty}$ can be controlled by  $\|v\|_{L^\infty}$, and $\|v\|_{L^\infty}$ can be controlled by  $\|u\|_{L^\infty}$ in turn. In the proof, we also use the classical existence theorem and the representation of solutions for the Dirichlet problems.

\medskip

For any $R>0$, recall the Green function $G_{2, R}(x, y)$ for the $-\Delta$ in the ball $B_{R}(0)\subset \mathbb{R}^n$ with $n\geq 3$:
\begin{equation*}
G_{2, R}(x, y)=\frac{1}{(2-n)\omega_n}\left(|y-x|^{2-n}-\left|\frac{|x|}{R}y-\frac{R}{|x|}x\right|^{2-n}\right), \quad  x\in B_R(0)\backslash \{0\},\,\, y\in B_R(0)\backslash \{x\},
\end{equation*}
and
\begin{equation}\label{eq3.1}
G_{2, R}(0,y)=\frac{1}{(2-n)\omega_n}\left(|y|^{2-n}-R^{2-n}\right), \quad \forall \, y\in  B_R(0)\backslash \{0\},
\end{equation}
where $\omega_n$ denotes the area of the unit sphere in $\mathbb{R}^{n}$.

If $n\geq3$ is odd, then the Green function $G_{n-1, R}(x, y)$ for $(-\Delta)^\frac{n-1}{2}$ on the ball $B_{R}(0)$ is given by
\begin{equation}\label{eq3.2}
\begin{split}
&G_{n-1, R}(x, y)\\
=&\int_{B_R(0)}\left(\int_{B_R(0)}\left(\cdots\left(\int_{B_R(0)}\left(\int_{B_R(0)}G_{2, R}(x, z^1) G_{2, R}(z^1, z^2)dz^1   \right)       G_{2, R}(z^2, z^3)dz^2  \right)\cdots\right)
\right.\\
&\left.
\times G_{2, R}(z^{\frac{n-5}{2}}, z^\frac{n-3}{2})dz^{\frac{n-5}{2}}\right)G_{2, R}(z^{\frac{n-3}{2}}, y)dz^{\frac{n-3}{2}}.
\end{split}
\end{equation}
Let
\begin{equation}\label{eq3.2-1}
\phi_R(x)=
\begin{cases}
(R^2-|x|^2)^{\frac{1}{2}},&x\in B_{R}(0),\\
0,& x\in \mathbb{R}^n\backslash B_{R}(0).
\end{cases}
\end{equation}
It is well-known that (see \cite{RS})
\begin{equation}\label{eq3.3}
(-\Delta)^{\frac{1}{2}}\phi_R(x)=c_0, \quad \forall \, x\in B_{R}(0)
\end{equation}
with $c_0$ independent of $R$ and  easy to check that $\phi_R(x)$ is H\"{o}lder continuous on $\partial B_{R}(0).$
Therefore, by the classical existence theorem for the  Poisson equation (see e.g. \cite{GT}) and an iteration process, there exists a nonnegative function $h_{R}(x)$ such that
\begin{equation}\label{a1}
  (-\Delta)^{i} h_{R}(x)\geq 0,\,\,\, i=0,\cdots, \frac{n-3}{2}, \quad \forall \, x\in B_{R}(0),
\end{equation}
and
\begin{equation}\label{a2}
  (-\Delta)^{\frac{n-1}{2}} h_{R}(x)= \phi_R(x), \quad \forall \, x\in B_{R}(0), \quad \text{if} \,\, n\,\, \mbox{is odd}.
\end{equation}
Indeed, $h_{R}(x)$ can be represented via the Green function $G_{n-1, R}(x, y)$  in \eqref{eq3.2} as
\begin{equation}\label{eq3.3-1}
h_{R}(x)=\int_{B_{R}(0)} G_{n-1, R}(x, y)\phi_R(y)\mathrm{d}y,
\end{equation}
moreover, by \eqref{eq3.1}, we have
\begin{equation}\label{eq3.3-2}
h_{R}(0)=\int_{B_{R}(0)} G_{n-1, R}(0, y)\phi_R(y)\mathrm{d}y=R^{n}\int_{B_{1}(0)} G_{n-1, 1}(0, y)\phi_1(y)\mathrm{d}y:=K_1R^{n},
\end{equation}
where $K_1$ is a positive constant depending only on $n$ which is defined by
\begin{equation}\label{eq3.4}
K_1:=\int_{B_{1}(0)} G_{n-1, 1}(0, y)\phi_1(y)\mathrm{d}y.
\end{equation}

\medskip

If $n\geq4$ is even, we define
\begin{equation}\label{eq3.5}
K_2:=\int_{B_1(0)} G_{n, 1}(0, y)\mathrm{d}y,
\end{equation}
where the Green function $G_{n, R}(x, y)$ for $(-\Delta)^\frac{n}{2}$ on the ball $B_{R}(0)$ can be represented as
\begin{equation*}
\begin{split}
&G_{n, R}(x, y)\\
=&\int_{B_R(0)}\left(\int_{B_R(0)}\left(\cdots\left(\int_{B_R(0)}\left(\int_{B_R(0)}G_{2, R}(x, z^1) G_{2, R}(z^1, z^2)dz^1\right)G_{2, R}(z^2, z^3)dz^2  \right)\cdots\right)
\right.\\
&\left.
\times G_{2,R}(z^{\frac{n}{2}-2}, z^{\frac{n}{2}-1})dz^{\frac{n}{2}-2}\right)G_{2, R}(z^{\frac{n}{2}-1}, y)dz^{\frac{n}{2}-1}.
\end{split}
\end{equation*}
It is easy to check that $K_2$ is a positive constant depending only on $n$.

\medskip

Let
\begin{equation}\label{eq3.6}
\sigma:=
\begin{cases}
\min\left\{\frac{1}{4}, \left(\frac{c_0}{2K_1}\right)^{\frac{1}{n}}\right\},&\mbox{if }\,\, n\,\, \mbox{is odd}; \\
\min\left\{\frac{1}{4}, \left(\frac{1}{2K_2}\right)^{\frac{1}{n}}\right\},&\mbox{if }\,\, n\,\, \mbox{is even};
\end{cases}
\end{equation}
where $c_0$, $K_1$ and $K_{2}$ are presented in \eqref{eq3.3}, \eqref{eq3.4} and \eqref{eq3.5} respectively.

\medskip

Without loss of generality, we may assume that
$$\Omega \subset B_{\sigma}(0)$$
by rescaling, where $\sigma$ is defined in \eqref{eq3.6}.
In fact, let $\rho:=\rho(\Omega)\geq1$ be the smallest radius such that $\Omega \subset B_{\sigma\rho}(0)$. Let
$$
u_{\rho}(x)=\rho^{\frac{n(p+1)}{pq-1}}u(\rho x)\,\,\, \mbox{and}\,\,\, v_{\rho}(x)=\rho^{\frac{n(q+1)}{pq-1}}v(\rho x)
$$
Then
$$
(-\Delta)^{\frac{n}{2}} u_{\rho}(x)= v_\rho^p(x) \,\,\, \mbox{and}\,\,\,(-\Delta)^{\frac{n}{2}} v_{\rho}(x)= u_\rho^q(x),
$$
it follows that $(u_\rho(x), v_\rho(x))$ is also a pair of solution to \eqref{eq1-1} in $\rho^{-1}\Omega \subset B_{\sigma}(0),$
and thus we only need to consider  $(u_\rho(x), v_\rho(x))$ instead.

\medskip

In what follows, we shall denote the maxima of $u$ and $v$ respectively by
$$
M:=\max_{\overline{\Omega}} u \quad \mbox{and} \quad N:=\max_{\overline{\Omega}} v.
$$

By rescaling and constructing auxiliary functions, we can derive some preliminary relationships between $M$  and $N.$

\begin{lemma}\label{le3.2}
Let $n\geq3$ and $\Omega\subset \mathbb{R}^n$ be a bounded domain with smooth boundary and $(u, v)$ be a pair of positive classical solution to \eqref{eq1-1}.
Then
$$
M\leq \frac{1}{2} N^p \quad \mbox{and} \quad N\leq \frac{1}{2} M^q.
$$
\end{lemma}
\begin{proof}
Denote
$$
\tilde u(x)=\frac{u(x)}{M} \quad \mbox{and} \quad \tilde v(x)=\frac{v(x)}{N}.
$$
Then $\|\tilde u\|_{L^\infty(\Omega)}=\|\tilde v\|_{L^\infty(\Omega)}=1,$
$$
(-\Delta)^{\frac{n}{2}} \tilde u(x)=\frac{1}{M}(-\Delta)^{\frac{n}{2}} u(x) =\frac{N^p}{M}\tilde v^p(x),
$$
and
$$
(-\Delta)^{\frac{n}{2}} \tilde v(x)=\frac{1}{N}(-\Delta)^{\frac{n}{2}} v(x) =\frac{M^q}{N}\tilde u^q(x).
$$

Without loss of generality, we may assume that $\tilde u(x)$ attains its maximum at $0\in \Omega$ such that $\tilde u(0)=1.$

\smallskip

If $n\geq3$ is odd, taking $R=\sigma$ in \eqref{eq3.2}, \eqref{eq3.2-1} and \eqref{eq3.3-1}, we have
\begin{equation*}
h_{\sigma}(x)=\int_{B_{\sigma}(0)} G_{n-1, \sigma}(x, y)\phi_\sigma(y)\mathrm{d}y,
\end{equation*}
and $h_{\sigma}(0)=K_1\sigma^{n}$ by \eqref{eq3.3-2}, where
$$
\phi_\sigma(x)=
\begin{cases}
(\sigma^2-|x|^2)^{\frac{1}{2}},&x\in B_{\sigma}(0),\\
0,& x\in \mathbb{R}^n\backslash B_{\sigma}(0).
\end{cases}
$$
Then by \eqref{a1} and \eqref{a2}, we know that
$$
(-\Delta)^{i} h_{\sigma}(x)\geq 0, \quad i=1,\cdots, \frac{n-3}{2}, \quad \forall \,\, x\in B_{\sigma}(0),
$$
and
$$
(-\Delta)^{\frac{n-1}{2}} h_{\sigma}(x)= \phi_{\sigma}(x), \quad \forall \,\, x\in B_{\sigma}(0).
$$
Now we compare $C_{1}h_{\sigma}(x)$ with $\tilde u(x)$, where $C_1=\frac{N^p}{c_0M}$ and $c_0$ is presented in \eqref{eq3.3}. It is obvious that
$$
C_1(-\Delta)^{\frac{n}{2}}h_{\sigma}(x)=C_1(-\Delta)^{\frac{1}{2}}\circ(-\Delta)^{\frac{n-1}{2}} h_{\sigma}(x)=C_1(-\Delta)^{\frac{1}{2}} \phi_\sigma(x)=\frac{N^p}{M}\geq(-\Delta)^{\frac{n}{2}}\tilde u(x).
$$
Noting that
$$
(-\Delta)^{i} \tilde u(x)= 0, \,\,\, i=0,\cdots, \frac{n-1}{2}, \quad \forall \, x\in \mathbb{R}^n\setminus\Omega,
$$
it follows from the maximum principle and the definition of $\sigma$ in \eqref{eq3.6} that
$$
1=\tilde u(0)\leq C_1h_{\sigma}(0)=\frac{N^p}{c_0M}K_1 \sigma^{n}\leq \frac{1}{2}\frac{N^p}{M}
$$
and hence
$$M\leq \frac{1}{2} N^p.$$
Exchanging $u$ and $v$ yields the other estimate
$$N\leq \frac{1}{2} M^q.$$

\smallskip

If $n\geq4$ is even, we consider the following auxiliary function defined on $B_{\sigma}(0)$:
$$
g_{\sigma}(x):=\frac{N^p}{M}\int_{B_{\sigma}(0)} G_{n, \sigma}(x,y)\mathrm{d}y,
$$
where $\sigma$ is defined in \eqref{eq3.6}. By rescaling, we can show that
$$
g_{\sigma}(0)=K_2\sigma^{n}\frac{N^p}{M}
$$
with $K_2$ defined in \eqref{eq3.5}.

We compare the function $g_{\sigma}(x)$ with $\tilde u(x)$. Since
$$(-\Delta)^i g_{\sigma}(x)\geq 0=(-\Delta)^i\tilde u(x), \quad \forall x\in  \partial\Omega, \,\,\, i=0,\cdots, \frac{n}{2}-1,$$
and
$$
(-\Delta)^{\frac{n}{2}}g(x)=\frac{N^p}{M}\geq (-\Delta)^{\frac{n}{2}}\tilde u(x), \quad  \forall \, x\in\Omega,
$$
by the maximum principle and the definition of $\sigma$ in \eqref{eq3.6}, we obtain that
$$
1=\tilde u(0)\leq g_{\sigma}(0)=K_2\sigma^{n}\frac{N^p}{M}\leq \frac{1}{2}\frac{N^p}{M},
$$
and hence
$$M\leq\frac{1}{2} N^p.$$
Exchanging $u$ and $v$ yields the other estimate
$$N\leq \frac{1}{2} M^q.$$
This concludes the proof of Lemma \ref{le3.2}.
\end{proof}

\section{Uniform a priori estimates}

In this section, we will carry out our proof of Theorem \ref{th1.3} for $n\geq2$ and $pq-1\geq\kappa$ with $\kappa>0$. Under the assumption \eqref{condition1} (i.e., $p\leq Kq$ with $K\geq1$), we establish the uniform a priori estimates for the component $u$ of $(u,v)$. We can also prove in a similar way that $v$ is uniformly bounded under the assumption $pq-1\geq\kappa$ and $q\leq Kp$ for some $\kappa>0$ and $K\geq1$.

\medskip

By Lemma \ref{le3.2}, we may assume that
\begin{equation}\label{a4}
  M, N>1 \quad \text{and} \quad M^{q}, \, N^{p}>\Gamma(n,\Omega,\kappa)
\end{equation}
with the sufficiently large uniform quantity $\Gamma(n,\Omega,\kappa)$ depending only on $n$, $\Omega$ and $\kappa$ (to be determined later). Furthermore, we may assume
\begin{equation}\label{a5}
  M>\Lambda(n,\Omega,\kappa)2^{\frac{n^{2}K}{2}}
\end{equation}
with the sufficiently large uniform quantity $\Lambda(n,\Omega,\kappa)$ depending only on $n$, $\Omega$ and $\kappa$ (to be determined later), or else we have done.

\medskip

By proving uniform estimates on integrals (with logarithmic singularity) and using Green's representation of the higher order derivative of $(u,v)$ and the Harnack inequality, we first establish more precise relationships between $M$ and $N$, which indicates that $\|u\|_{L^\infty}$ and $\|v\|_{L^\infty}$ are comparable and equivalent. Then, we prove more precise estimates for $\|-\Delta u\|_{L^\infty}$ and $\|-\Delta v\|_{L^\infty}$ in terms of $\|u\|_{L^\infty}$ and $\|v\|_{L^\infty}$ by Green's representation of the higher order derivative of $(u,v)$ and the Harnack inequality, in which we carefully divided the integral domain into three proper parts. Finally, we are able to use the Green's representation for $u$ and $v$ to derive the uniform bounds.

\medskip

To obtain the uniform a priori estimates for $u$ and $v$, we need to estimate $\int_\Omega v^p\mathrm{d}x$ and $\int_\Omega u^q\mathrm{d}x$.
\begin{lemma}\label{le3.1}
Let $n\geq3$ and $(u,v)$ be a pair of positive classical solution to \eqref{eq1-1} in a bounded and strictly convex domain $\Omega\subset\mathbb{R}^n$ with smooth boundary. Assume that $pq-1\geq\kappa$. Then there exist some positive constants $\delta$ ($\leq\min\left\{\sigma,\left(\frac{n}{2^{n}\omega_{n}}\right)^{\frac{1}{2n}}\right\}$) depending only on $n$ and $\Omega$, and $\gamma_1$ depending only on $n, \Omega$ and $\kappa$ such that

\noindent (i) The maxima of $(-\Delta)^{i}u$ and $(-\Delta)^{i}v$ ($i=0,\cdots,\lfloor\frac{n}{2}\rfloor$) in $\overline{\Omega}$ are attained in $\Omega_{\delta}:=\{x\in\Omega \mid dist(x, \partial \Omega)\geq \delta\}$, where $\lfloor x\rfloor$ denotes the largest integer $<x$;

\noindent (ii) The solution $(u, v)$ satisfies the uniform bound estimates:
$$
\int_\Omega v^p(x)\mathrm{d}x\leq \gamma_1,
$$
and
$$
\int_\Omega u^q(x)\mathrm{d}x\leq \gamma_1.
$$

\end{lemma}
\begin{proof}
By Theorem \ref{th1.1}, Theorem \ref{th1.2} and Heine-Borel theorem, it is easy to know that $(-\Delta)^{i}u$ and $(-\Delta)^{i}v$ ($i=0,\cdots,\lfloor\frac{n}{2}\rfloor$) are strictly monotone increasing along the inward normal direction in a boundary layer with uniform thickness $\delta$ depending only on $n$ and $\Omega$ (see \cite{DD}). Therefore, the maxima of $(-\Delta)^{i}u$ and $(-\Delta)^{i}v$ ($i=0,\cdots,\lfloor\frac{n}{2}\rfloor$) must be attained in
$$\Omega_\delta:=\left\{x\in\Omega \mid dist(x, \partial \Omega)\geq \delta\right\}.$$
Hence we have reached (i).

\medskip
Next we prove (ii).
Let $\eta(x) \in C_0^\infty (\Omega)$, $\eta(x) \in [0, 1], \; x\in \Omega$ and
\begin{equation*}
\eta(x)=
\begin{cases}
1,\; &x\in \Omega_\delta,\\
0,\; &x\in \Omega \backslash \Omega_{\frac{\delta}{2}}.
\end{cases}
\end{equation*}

By Theorem \ref{th1.1}, Theorem \ref{th1.2} and a similar argument as in \cite{DD}, we can derive
\begin{equation}\label{eq3-1}
\int_\Omega u^q(x) \mathrm{d}x \leq C_1 \int_{\Omega_\delta}  u^q(x) \mathrm{d}x,
\end{equation}
and
\begin{equation}\label{eq3-2}
\int_\Omega v^p(x) \mathrm{d}x \leq C_1 \int_{\Omega_\delta}  v^p(x) \mathrm{d}x,
\end{equation}
where $C_{1}$ depends only on $n$ and $\Omega.$

Using the integrating by parts (if $n$ is odd, $u\in C^{n-2}_{0}(\Omega)$, c.f. \cite[Theorem 2]{CW}), we obtain
 $$
\int_{\Omega} (-\Delta)^{\frac{n}{2}} u(x) \eta(x) \mathrm{d}x=\int_{\Omega} u(x) (-\Delta)^{\frac{n}{2}}\eta(x) \mathrm{d}x.
$$
Then from the first equation in \eqref{eq1-1}, one has
\begin{equation*}
\int_{\Omega}v^p(x) \eta(x) \mathrm{d}x=\int_{\Omega} u(x) (-\Delta)^{\frac{n}{2}}\eta(x) \mathrm{d}x,
\end{equation*}
which together with \eqref{eq3-2} yields
\begin{equation}\label{eq3-3}
\begin{split}
\int_{\Omega}v^p(x) \mathrm{d}x \leq& C_1 \int_{\Omega_\delta}v^p(x) \mathrm{d}x \\
\leq& C_1 \int_{\Omega}v^p(x) \eta(x) \mathrm{d}x\\
=& C_1\int_{\Omega} u(x) (-\Delta)^{\frac{n}{2}}\eta(x) \mathrm{d}x\\
\leq& C_2 \int_{\Omega} u(x)\mathrm{d}x\\
\leq& C_2 |\Omega|^{1-\frac{1}{q}}\left (\int_{\Omega} u^q(x)\mathrm{d}x\right)^\frac{1}{q}.
\end{split}
\end{equation}
Similarly, using \eqref{eq3-1} and the second equation in \eqref{eq1-1}, we have
\begin{equation}\label{eq3-4}
\int_{\Omega}u^q(x) \mathrm{d}x \leq
C_{2}|\Omega|^{1-\frac{1}{p}}\left (\int_{\Omega} v^p(x)\mathrm{d}x\right)^\frac{1}{p}.
\end{equation}
Combining \eqref{eq3-3} with \eqref{eq3-4} and using the assumption $pq-1\geq\kappa$, we end up with
$$
\int_\Omega u^q(x)\mathrm{d}x\leq C_{2}^{\frac{2}{1-\frac{1}{pq}}}|\Omega|\leq C_2^{\frac{2}{1-\frac{1}{\kappa+1}}}|\Omega|=:C_3(n,\Omega, \kappa)
$$
and
$$
\int_\Omega v^p(x)\mathrm{d}x\leq C_{2}^{\frac{2}{1-\frac{1}{pq}}}|\Omega|\leq C_2^{\frac{2}{1-\frac{1}{\kappa+1}}}|\Omega|=:C_3(n,\Omega, \kappa).
$$
Therefore, (ii) is valid and we complete the proof of Lemma \ref{le3.1}.
\end{proof}

For any fixed $x\in \Omega_\delta,$ let  $G(x, y)$ be  Green's function for $(-\Delta)^{\frac{n}{2}}$ with pole at $x$ associated with Navier boundary conditions (if $n$ is even) and
 Navier exterior conditions (if $n$ is odd). Then we have
$$
G(x, y)=C_n \ln\frac{1}{|x-y|}-h(x,y), \quad \forall \, y \in \overline{\Omega},
$$
where $h(x, y)$ denotes the $\frac{n}{2}$-harmonic function.

\medskip

If $n\geq 4$ is even, $h(x, y)$ satisfies
\begin{equation}\label{eq4-1}
\begin{cases}
(-\Delta)^{\frac{n}{2}} h(x,y)=0, &y\in \Omega, \\
(-\Delta)^i h(x,y) =(-\Delta)^i \left(C_n \ln\left(\frac{1}{|x-y|}\right)\right),\,\,  i=0, 1, \cdots, \frac{n}{2}-1, &y\in \partial \Omega.
\end{cases}
\end{equation}

\medskip

If $n\geq 3$ is odd, $h(x, y)$ satisfies
\begin{equation}\label{eq4-2}
\begin{cases}
(-\Delta)^{\frac{n}{2}} h(x,y)=0, &y\in \Omega, \\
(-\Delta)^i h(x,y) =(-\Delta)^i \left(C_n \ln\left(\frac{1}{|x-y|}\right)\right),\,\,  i=0, 1, \cdots, \frac{n-1}{2}, &y\in \mathbb{R}^n\backslash \Omega.
\end{cases}
\end{equation}

\medskip

We will show that the $\frac{n}{2}$-harmonic function $h(x, y)$ and $(-\Delta)^{i}h(x,y)$ ($i=1, \cdots, \left\lfloor\frac{n}{2}\right\rfloor$) are bounded for any $x\in \Omega_\delta$ and $y\in \Omega$.

\begin{lemma}\label{le4.1}
The $\frac{n}{2}$-harmonic function harmonic function $h(x, y)$ is bounded from above for any $x\in \Omega_\delta$ and $y\in \Omega,$ more precisely, there exists a constant $\gamma_2>0$ depending only on $n$ and $\Omega$ such that
$$
0\leq h(x, y)\leq \gamma_2, \quad \forall \, x\in \Omega_\delta, \,\,\, \forall \, y\in \Omega.
$$
Moreover,
$$
0 \leq (-\Delta)^i h(x, y) \leq \gamma_2, \quad i=1, \cdots, \left\lfloor\frac{n}{2}\right\rfloor, \quad \forall \, x\in \Omega_\delta, \,\,\, \forall \, y\in \Omega
$$
where $\left\lfloor x\right\rfloor$ denotes the largest integer $<x$.
\end{lemma}
\begin{proof}
Lemma \ref{le4.1} can be deduced from   \eqref{eq4-1} and  \eqref{eq4-2} satisfied by $h(x, y).$  Indeed,
if $n$ is odd, the conclusion follows from Lemma 4.1 in \cite{CW}. If $n$ is even, the result is a direct consequence of (2.20) and (2.21) in \cite{DD}.
\end{proof}

Based on the above lemmas, we are able to obtain the following refined relationships between $M$ and $N$, which is a key ingredient in our proof.
\begin{lemma}\label{le4.4}
Let $n\geq3$ and $(u,v)$ be a pair of positive classical solution to \eqref{eq1-1} with $pq-1\geq\kappa$. Then there exists a constant $\gamma_3>0$ depending only on $n$, $\Omega$ and $\kappa$ such that
$$
N\leq \gamma_3 M^{\frac{q}{\frac{2}{n}p+1}} \quad \mbox{and} \quad M\leq \gamma_3 N^{\frac{p}{\frac{2}{n}q+1}}.
$$
\end{lemma}
\begin{proof}
From the first equation of \eqref{eq1-1}, by Green's representation formula, we have
\begin{equation}\label{eq4-3}
u(x)=C_1 \int_\Omega \ln\frac{1}{|x-y|} v^p(y)\mathrm{d}y -\int_\Omega h(x, y) v^p(y)\mathrm{d}y.
\end{equation}
It can be deduced from \eqref{eq1-1} and the maximum principle that
$$
-\Delta u(x)\geq 0, \quad \forall \, x\in \Omega.
$$
Applying Theorem \ref{th1.1} and Theorem \ref{th1.2}, we know that the maximum of $-\Delta u(x)$ in $\Omega$ can only be attained at some point $x_1 \in \Omega_\delta$, more precisely,
$$
-\Delta u(x_1)=\max_{\overline{\Omega}}\left(-\Delta u(x)\right).
$$

It can be seen from Lemma \ref{le4.1} that
$$
-\Delta h(x_1,y)\geq 0, \quad \forall \, y\in \Omega,
$$
then by \eqref{eq4-3}, we have
\begin{equation}\label{eq4-4}
\begin{split}
-\Delta u(x_1)&=C_2\int_\Omega\frac{v^p (y)}{|x_1-y|^2} \mathrm{d}y -\int_\Omega \left[-\Delta h(x_1, y)\right] v^p(y)\mathrm{d}y\\
&\leq C_2\int_\Omega \frac{v^p (y)}{|x_1-y|^2} \mathrm{d}y \\
&=C_{2}\int_{B_{r_0}(x_1)} \frac{v^p (y)}{|x_1-y|^2} \mathrm{d}y+C_{2}\int_{\Omega \cap B_{r_0}^c(x_1)} \frac{v^p (y)}{|x_1-y|^2} \mathrm{d}y \\
&=: J_1+J_2,
\end{split}
\end{equation}
where $r_{0}:=\delta N^{-\frac{p}{n}}$. By Lemma \ref{le3.1}, we have the following estimates on $J_{1}$ and $J_{2}$:
\begin{equation}\label{eq4-6'}
J_1=C_{2}\int_{B_{r_0}(x_1)} \frac{v^p (y)}{|x_1-y|^2} \mathrm{d}y\leq C_{3}r_{0}^{n-2}N^p=C_{4}N^{\frac{2p}{n}},
\end{equation}
\begin{equation}\label{eq4-7'}
J_2=C_{2}\int_{\Omega \cap B_{r_{0}}^c(x_1)} \frac{v^p (y)}{|x_1-y|^2} \mathrm{d}y\leq C_{2}r_{0}^{-2}\int_{\Omega}v^p(y)\mathrm{d}y\leq C_{5} N^{\frac{2p}{n}}.
\end{equation}
Combining \eqref{eq4-4}, \eqref{eq4-6'} and \eqref{eq4-7'} yields that
$$
0\leq -\Delta u(x_1)=\max_{\overline{\Omega}}\left(-\Delta u(x)\right)\leq C_{6}N^{\frac{2p}{n}}.
$$
Similarly, there exists some point $x_2 \in \Omega_\delta$ such that $-\Delta v(x_2)=\max\limits_{\overline{\Omega}}\left(-\Delta v(x)\right)$, by exchanging $u$ and $v$, one can get
$$
0\leq -\Delta v(x_2)=\max_{\overline{\Omega}}\left(-\Delta v(x)\right)\leq C_{7}M^{\frac{2q}{n}}.
$$

To estimate $M-u$, assume that $M=\max\limits_{\overline{\Omega}} u(x)$ is attained at $x_u\in\Omega_{\delta}$. Since $-\Delta(M-u)=\Delta u$ and $B_\delta(x_u)\subset \Omega$, employing the Harnack inequality (refer to e.g. \cite{GT}, see also \cite{S}), we have
$$
\sup_{B_r(x_u)}(M-u(x))\leq C_8 \left(\inf_{B_{2r}(x_u)}(M-u(x))+ r \|\Delta u\|_{L^n(B_{2r}(x_u))}\right)
$$
for all $r\in(0, \frac{\delta}{4}]$. Taking
$$r=R_1=\frac{\delta}{4}\sqrt{\frac{M}{C_{6}C_{8}qN^{p}}},$$
it follows from Lemma \ref{le3.2} that $r\in(0, \frac{\delta}{4}]$ (we may assume that $C_6, \, C_8>1$) and thus we deduce that
 $$
0\leq M-u(x)\leq 2C_{8}\left(\frac{\omega_n}{n}\right)^{\frac{1}{n}}\left(-\Delta u(x_1)\right)r^2\leq\frac{M}{16q}, \quad \forall \, |x-x_u|\leq R_1,
$$
here we have used the fact that $\delta \leq\min\left\{\sigma,\left(\frac{n}{2^{n}\omega_{n}}\right)^{\frac{1}{2n}}\right\}$ in Lemma \ref{le3.1}.
Thus we infer that
\begin{equation}\label{eqspecial}
u^q(y)\geq M^q \left(1-\frac{1}{16 q}\right)^q\geq e^{-\frac{1}{15}} M^q, \quad \forall \, y\in B_{R_1}(x_u)\subset \Omega.
\end{equation}
Applying Lemma \ref{le3.1} and \eqref{eqspecial}, we get
$$
C_9 \geq \int_\Omega u^q(x)\mathrm{d}x \geq \int_{B_{R_1}(x_u)}u^q(x)\mathrm{d}x \geq e^{-\frac{1}{15}} M^q |B_{R_1}(x_u)|\geq C_{10}\frac{M^{q+\frac{n}{2}}}{q^\frac{n}{2} N^\frac{np}{2}},
$$
which gives the relationship between $M$ and $N$ as
$$
M\leq C_{11}q^{\frac{n}{2q+n}}N^{\frac{np}{2q+n}}\leq C_{12}N^{\frac{p}{\frac{2}{n}q+1}}.
$$
In a similar way, we can deduce that there exists a uniform positive constant $C_{13}$ such that
$$
N\leq C_{13}M^{\frac{q}{\frac{2}{n}p+1}}.
$$
This concludes our proof of Lemma \ref{le4.4}.
\end{proof}

From now on, we take the quantities $\Gamma(n,\Omega,\kappa):=\gamma_{3}^{n+2}\geq1$ and $\Lambda(n,\Omega,\kappa):=\gamma_{3}\geq1$ depending only on $n$, $\Omega$ and $\kappa$. Thus assumptions \eqref{a4} and \eqref{a5} become $M^{q},\,N^{p}>\Gamma(n,\Omega,\kappa):=\gamma_{3}^{n+2}$ and $M>\Lambda(n,\Omega,\kappa)2^{\frac{n^{2}K}{2}}:=\gamma_{3}2^{\frac{n^{2}K}{2}}$, where $\gamma_{3}$ is presented in Lemma \ref{le4.4}.

\medskip

From Lemma \ref{le4.4} and the assumption $p\leq Kq$, we deduce that
$$\gamma_{3}2^{\frac{n^{2}K}{2}}<M\leq \gamma_3 N^{\frac{p}{\frac{2}{n}q+1}}\leq \gamma_3 N^{\frac{Kq}{\frac{2}{n}q+1}}\leq\gamma_3 N^{\frac{nK}{2}},$$
and hence $N>2^{n}$. Similarly, we can also get $M>2^{n}$ provided that $q\leq Kp$ and $N>\gamma_{3}2^{\frac{n^{2}K}{2}}$. That is, if one of the maxima of $(u,v)$ is large enough, then the other is also large.

\medskip

We will show that $u$ is uniformly bounded provided that $p\leq Kq$. To this end, we need the following uniform estimates on integrals with logarithmic singularity.
\begin{lemma}\label{le4.2}
Let $(u,v)$ be a pair of positive classical solution to \eqref{eq1-1} with $pq-1\geq\kappa$. Then there exists a positive constant $\gamma_4$ depending only on $n$,  $\Omega$  and $\kappa$ such that
$$
\frac{1}{M}\int_\Omega \ln\frac{1}{|x-y|} v^p(y)\mathrm{d}y \leq \gamma_4.
$$
and
$$
\frac{1}{N}\int_\Omega \ln\frac{1}{|x-y|} u^q(y)\mathrm{d}y \leq \gamma_4.
$$
\end{lemma}
\begin{proof}
By Lemma \ref{le3.1}, Lemma \ref{le4.1}, and the Green representation formula, for any $x\in \Omega_\delta$, we have
\begin{equation*}
\begin{split}
M\geq u(x)&=\int_\Omega G(x, y) v^p(y) \mathrm{d}y\\
&=C_1 \int_\Omega \ln\frac{1}{|x-y|} v^p(y)\mathrm{d}y -\int_\Omega h(x, y) v^p(y)\mathrm{d}y\\
&\geq C_1 \int_\Omega \ln\frac{1}{|x-y|} v^p(y)\mathrm{d}y -C_2\int_\Omega v^p(y)\mathrm{d}y\\
&\geq C_1 \int_\Omega \ln\frac{1}{|x-y|} v^p(y)\mathrm{d}y -C_3.
\end{split}
\end{equation*}
As a result, we obtain
$$
\frac{1}{M}\int_\Omega \ln\frac{1}{|x-y|} v^p(y)\mathrm{d}y \leq \frac{C_3}{C_1 M}+\frac{1}{C_1}\leq C_4.
$$
By a similar argument, we can derive
$$
\frac{1}{N}\int_\Omega \ln\frac{1}{|x-y|} u^q(y)\mathrm{d}y \leq C_5.
$$
This finishes the proof of Lemma \ref{le4.2}.
\end{proof}

To obtain the uniform a priori estimates for $u$ and $v$, we establish more precise $L^\infty$ estimates for $-\Delta u$ and $-\Delta v$ in terms of the maxima of $u$ and $v$, and then use the Harnack inequality to estimate how fast the values of $u$ and $v$ can decrease away from their maxima.
\begin{lemma}\label{le4.3}
Let $n\geq3$ and $(u,v)$ be a pair of positive classical solution to \eqref{eq1-1} with $pq-1\geq\kappa$. Assume further that
\begin{equation}\label{assumption}
p\leq Kq, \,\,\,\, M>\gamma_{3}2^{\frac{n^{2}K}{2}} \quad \text{or} \quad q\leq Kp, \,\,\,\, N>\gamma_{3}2^{\frac{n^{2}K}{2}},
\end{equation}
then there exist some constants $\gamma_5>0$, $0<\gamma_6<\frac{\delta}{4}$ depending only on $n$, $\Omega$ and $\kappa$ such that
$$
\max_{\overline{\Omega}}\;\mid-\Delta u(x)\mid \leq \gamma_5 M^{1-\frac{2}{n}} N^{\frac{2p}{n}}q^{\frac{2}{n}-1}
$$
and
$$
\max_{\overline{\Omega}}\;\mid-\Delta v(x)\mid \leq \gamma_5 N^{1-\frac{2}{n}} M^{\frac{2q}{n}}p^{\frac{2}{n}-1}.
$$
Moreover,
$$
0\leq M-u(x)\leq  \frac{M}{16q}, \quad \forall \, |x-x_u|\leq R_1
$$
and
$$
0\leq N-v(x)\leq  \frac{N}{16p}, \quad \forall \, |x-x_v|\leq R_2
$$
with $R_1:=\gamma_6 M^\frac{1}{n}N^{-\frac{p}{n}}q^{-\frac{1}{n}}$ and $R_2:=\gamma_6 N^\frac{1}{n}M^{-\frac{q}{n}}p^{-\frac{1}{n}}$, where $u(x_{u})=M$ and $v(x_{v})=N$.
\end{lemma}
\begin{proof}
We prove Lemma \ref{le4.3} under the assumption $p\leq Kq$ and $M>\gamma_{3}2^{\frac{n^{2}K}{2}}$, the case $q\leq Kp$ and $N>\gamma_{3}2^{\frac{n^{2}K}{2}}$ can be handled in a similar way by simply exchanging $u$ and $v$. Denote
$$
r_1=\delta 2^{\frac{1}{n}}M^{\frac{1}{n}}N^{-\frac{p}{n}}q^{-\frac{1}{n}} \quad \mbox{and} \quad r_2=\delta M^{\frac{1}{n}-\frac{1}{2}} N^{-\frac{p}{n}}q^{\frac{n-2}{2n}}.
$$
By Lemma \ref{le3.1}, the maximum of $-\Delta u(x)$ in $\Omega$ can only be attained at some point $x_{1}\in\Omega_{\delta}$. Since Lemma \ref{le3.2} yields $M\leq \frac{1}{2} N^p$, then one has $r_1\leq \delta$ and $B_{r_1}(x_1)\subset\Omega$. Moreover, by Lemma \ref{le4.1} and \eqref{eq4-3}, in the case $r_1<r_2$, we have
\begin{equation}\label{eq4-5}
\begin{split}
&-\Delta u(x_1)=\max_{\overline{\Omega}}\left(-\Delta u(x)\right)\\
\leq &C_{0}\int_\Omega \frac{v^p (y)}{|x_1-y|^2} \mathrm{d}y\\
=& C_{0}\int_{B_{r_1}(x_1)} \frac{v^p (y)}{|x_1-y|^2} \mathrm{d}y+C_{0}\int_{\Omega \cap B_{r_2}^c(x_1)} \frac{v^p (y)}{|x_1-y|^2} \mathrm{d}y+C_{0}\int_{\left(\bar B_{r_2}(x_1)\setminus B_{r_1}(x_1)\right)\cap\Omega}\frac{v^p (y)}{|x_1-y|^2} \mathrm{d}y\\
=:& I_1+I_2+I_3.
\end{split}
\end{equation}
Next we estimate $I_1,\,I_2$ and $I_3$ respectively (if $r_1\geq r_2$, then we only need to take $I_{1}$ and $I_{2}$ into account). First, we estimate $I_1$:
\begin{equation}\label{eq4-6}
I_1=C_{0}\int_{B_{r_1}(x_1)} \frac{v^p (y)}{|x_1-y|^2} \mathrm{d}y\leq C_1r_1^{n-2}N^p=C_2M^{1-\frac{2}{n}} N^{\frac{2p}{n}}q^{\frac{2}{n}-1}.
\end{equation}
Second, in order to estimate $I_2$, we use Lemma \ref{le3.1} to derive
\begin{equation}\label{eq4-7}
I_2=C_{0}\int_{\Omega \cap B_{r_2}^c(x_1)} \frac{v^p (y)}{|x_1-y|^2} \mathrm{d}y\leq C_{0}r_2^{-2}\int_\Omega v^p(y)\mathrm{d}y\leq C_3 M^{1-\frac{2}{n}} N^{\frac{2p}{n}}q^{\frac{2}{n}-1}.
\end{equation}
Finally, since $M>\Lambda(n,\Omega,\kappa)2^{\frac{n^{2}K}{2}}$ with $\Lambda(n,\Omega,\kappa):=\gamma_{3}$, we deduce from Lemma \ref{le4.4} that
$$
2\left(\frac{1}{2}-\frac{1}{n}\right)\ln q<\frac{nK}{2}\left(\frac{2}{n}q+1\right)\ln 2<\frac{p}{n}\ln N,
$$
it follows that $r_2\leq \delta$ and $B_{r_2}(x_1)\subset\Omega$. Hence, in the case $r_1<r_2,$ by Lemma \ref{le4.2}, $I_3$ can be estimated as follows:
\begin{equation}\label{eq4-8}
\begin{split}
I_3&=C_{0}\int_{\left(\bar B_{r_2}(x_1)\setminus B_{r_1}(x_1)\right)\cap\Omega}\frac{v^p (y)}{|x_1-y|^2} \mathrm{d}y\\
&\leq C_{0}\left(\frac{1}{M}\int_\Omega \ln\frac{1}{|x-y|} v^p(y)\mathrm{d}y \right)\frac{M}{r_1^2 \ln \frac{1}{r_2}}\\
&\leq C_4 M^{1-\frac{2}{n}}N^{\frac{2p}{n}}q^{\frac{2}{n}-1}.
\end{split}
\end{equation}
Combining \eqref{eq4-5},\eqref{eq4-6}, \eqref{eq4-7} and  \eqref{eq4-8} yields that
$$
0\leq -\Delta u(x_1)=\max_{\overline{\Omega}}\left(-\Delta u(x)\right)\leq C_5 M^{1-\frac{2}{n}}N^{\frac{2p}{n}}q^{\frac{2}{n}-1}.
$$
The estimate for the maximum of $-\Delta v(x)$ can be derived in a similar way by exchanging $u$ and $v$, we only need to notice that $2\left(\frac{1}{2}-\frac{1}{n}\right)\ln p<\frac{n}{2}p\ln2\leq\frac{nK}{2}q\ln2\leq\frac{1}{n}\ln\left[\left(\gamma_{3}2^{\frac{n^{2}K}{2}}\right)^{q}\right]<\frac{q}{n}\ln M$ due to the assumption $p\leq Kq$.

\medskip

Now assume that $M=\max\limits_{\overline{\Omega}}u(x)$ is attained at $x_u\in\Omega_{\delta}$ and $N=\max\limits_{\overline{\Omega}}v(x)$ is attained at $x_v\in\Omega_{\delta}$. Since $-\Delta(M-u)=\Delta u$ and $B_\delta(x_u)\subset \Omega$, employing the Harnack inequality (see e.g. \cite{GT}), we have
$$
\sup_{B_r(x_u)}\left(M-u(x)\right)\leq C_6\left(\inf_{ B_{2r}(x_u)}(M-u(x))+ r \|\Delta u\|_{L^n(B_{2r}(x_u))}\right),
$$
for all $r\in(0, \frac{\delta}{4}].$  Taking
$$r=R_1=\frac{1}{\sqrt{\bar{c}}}\frac{\delta}{4}M^\frac{1}{n}N^{-\frac{p}{n}}q^{-\frac{1}{n}}$$
 with
$\bar c=C_{5}C_{6}+1$,
then by Lemma \ref{le3.2} and \eqref{condition1}, we have $0<R_1<\frac{\delta}{4}$ (here we may assume that $C_5, C_6>1$) and therefore,
$$
0\leq M-u(x)\leq 2C_6\left(\frac{\omega_{n}}{n}\right)^{\frac{1}{n}}\left(-\Delta u(x_1)\right)R_1^2\leq \frac{M}{16q}, \quad \forall \,\, |x-x_u|\leq R_1.
$$
where we have used the fact that $\delta \leq\min\left\{\sigma,\left(\frac{n}{2^{n}\omega_{n}}\right)^{\frac{1}{2n}}\right\}$ in Lemma \ref{le3.1}.

Exchanging $u$ and $v$ yields the other estimate:
\begin{eqnarray*}
  && 0\leq N-v(x)\leq 2C_6\left(\frac{\omega_{n}}{n}\right)^{\frac{1}{n}}\max_{\overline{\Omega}}\left(-\Delta v(x)\right)R_2^2 \\
  && \qquad\qquad\qquad =2C_6\left(\frac{\omega_{n}}{n}\right)^{\frac{1}{n}}\left(-\Delta v(x_2)\right)R_2^2\leq \frac{N}{16p}, \quad \forall \,\, |x-x_v|\leq R_2,
\end{eqnarray*}
where $R_{2}=\frac{1}{\sqrt{\bar{c}}}\frac{\delta}{4}N^\frac{1}{n}M^{-\frac{q}{n}}p^{-\frac{1}{n}}$. This concludes our proof of Lemma \ref{le4.3}.
\end{proof}

\begin{remark}\label{rem1}
Let $n\geq3$ and $(u,v)$ be a pair of positive classical solution to \eqref{eq1-1} with $pq-1\geq\kappa$. Based on Lemma \ref{le4.3}, one can actually obtain more precise relationships between $M$ and $N$ than Lemma \ref{le4.4} under assumption \eqref{assumption}. That is, assume further that
\begin{equation}\label{assumption1}
p\leq Kq, \,\,\,\, M>\gamma_{3}2^{\frac{n^{2}K}{2}} \quad \text{or} \quad q\leq Kp, \,\,\,\, N>\gamma_{3}2^{\frac{n^{2}K}{2}},
\end{equation}
then there exists a constant $\gamma_7>0$ depending only on $n$, $\Omega$ and $\kappa$ such that
$$
N\leq \gamma_7 M^{\frac{q}{p+1}} \quad \mbox{and} \quad M\leq \gamma_7 N^{\frac{p}{q+1}}.
$$
Indeed, by Lemma \ref{le4.3}, we know that near the maximum point $x_{u}$ of $u(x)$,
\begin{equation}\label{eq4-9}
u^q(y)\geq M^q \left(1-\frac{1}{16 q}\right)^q\geq e^{-\frac{1}{15}} M^q, \quad \forall \,\, y\in B_{R_1}(x_u)\subset \Omega.
\end{equation}
Analogously,  we deduce that near the maximum point $x_{v}$ of $v(x)$,
\begin{equation}\label{a3}
  v^p(y)\geq N^p \left(1-\frac{1}{16 p}\right)^p\geq e^{-\frac{1}{15}} N^p, \quad \forall \,\, y\in B_{R_2}(x_v)\subset \Omega,
\end{equation}
where $R_1$ and $R_2$ are presented in Lemma \ref{le4.3}. Applying Lemma \ref{le3.1} and \eqref{eq4-9}, we get
$$
C_1 \geq \int_\Omega u^q(x)\mathrm{d}x \geq \int_{B_{R_1}(x_u)}u^q(x)\mathrm{d}x \geq e^{-\frac{1}{15}} M^q |B_{R_1}(x_u)|\geq C_2\frac{M^{q+1}}{q N^p},
$$
which gives the relationship between $M$ and $N$ as
$$
M\leq C_3 N^{\frac{p}{q+1}}.
$$
By a similar argument, we can also obtain from Lemma \ref{le3.1} and \eqref{a3} that
$$
N\leq C_4 M^{\frac{q}{p+1}}.
$$
\end{remark}

\begin{remark}\label{rem2}
It is clear from the proof of Lemma \ref{le4.3} that the conclusions in Lemma \ref{le4.3} are still valid provided that
\begin{equation}\label{assumption-a}
p\leq K^{q}, \,\,\,\, M>\gamma_{3}K^{n-2} \quad \text{or} \quad q\leq K^{p}, \,\,\,\, N>\gamma_{3}K^{n-2}
\end{equation}
for some $K\geq 2^{n}$. Indeed, assume that $p\leq K^{q}$ and $M>\gamma_{3}K^{n-2}$, then it follows from the estimates $2\left(\frac{1}{2}-\frac{1}{n}\right)\ln q<\frac{n-2}{n}\left(\frac{2}{n}q+1\right)\ln K<\frac{p}{n}\ln N$ and $2\left(\frac{1}{2}-\frac{1}{n}\right)\ln p<\frac{n-2}{n}\ln\left[\left(\gamma_{3}^{\frac{1}{n-2}}K\right)^{q}\right]<\frac{q}{n}\ln M$ that the proof of Lemma \ref{le4.3} still work.

Consequently, under the assumption \eqref{assumption-a}, the conclusions in Remark \ref{rem1} are also valid, that is,
$$
N\leq \gamma_7 M^{\frac{q}{p+1}} \quad \mbox{and} \quad M\leq \gamma_7 N^{\frac{p}{q+1}}.
$$
\end{remark}

Finally, we apply Green's representation formula and relationships between $M$ and $N$ (see Lemma \ref{le4.4} and Remark \ref{rem1}) to establish the uniform a priori estimates for $u$ and $v$.

\begin{lemma}\label{le4.5}
Let $n\geq3$ and $(u,v)$ be a pair of positive classical solution to \eqref{eq1-1} with $pq-1\geq\kappa$. Suppose $M^{q},\,N^{p}>\Gamma(n,\Omega,\kappa):=\gamma^{n+2}_{3}$. Assume further that
\begin{equation}\label{assumption2}
p\leq Kq, \,\,\,\, M>\gamma_{3}2^{\frac{n^{2}K}{2}} \quad \text{or} \quad q\leq Kp, \,\,\,\, N>\gamma_{3}2^{\frac{n^{2}K}{2}},
\end{equation}
where $\gamma_3$ is presented in Lemma \ref{le4.4}. Then there exists a positive constant $\gamma_8$ depending only on $n$, $\Omega$ and $\kappa$ such that
 $$
 M\leq \gamma_8 \quad \mbox{and} \quad N\leq \gamma_8.
 $$
\end{lemma}
\begin{proof}
We prove Lemma \ref{le4.5} under the assumption $p\leq Kq$ and $M>\gamma_{3}2^{\frac{n^{2}K}{2}}$, the case $q\leq Kp$ and $N>\gamma_{3}2^{\frac{n^{2}K}{2}}$ can be handled in a similar way by simply exchanging $u$ and $v$. Recall that Lemma \ref{le4.4} yields $N>2^{n}$.

First, using the Green representation formula, Lemma \ref{le3.1} and Lemma \ref{le4.1} yield
\begin{equation}\label{eq4-10}
\begin{split}
M\geq u(x_v)&=C_1\int_\Omega \ln \frac{1}{|x_v-y|} v^p(y)\mathrm{d}y-\int_\Omega h(x_v, y)v^p(y)\mathrm{d}y\\
&\geq C_1\int_\Omega \ln \frac{1}{|x_v-y|} v^p(y)\mathrm{d}y-\gamma_2\int_\Omega v^p(y)\mathrm{d}y\\
&\geq C_1\int_\Omega \ln \frac{1}{|x_v-y|} v^p(y)\mathrm{d}y-\gamma_{1}\gamma_2,
\end{split}
\end{equation}
where $x_{u}$ and $x_{v}$ are maximum points of $u$ and $v$ respectively. Then we utilize Lemma \ref{le4.3} and Lemma \ref{le4.4} to estimate the integral term
$$C_1\int_\Omega \ln \frac{1}{|x_v-y|} v^p(y)\mathrm{d}y$$
with the logarithmic singularity in \eqref{eq4-10} as follows:
\begin{equation}\label{eq4-11}
\begin{split}
 C_1\int_\Omega \ln \frac{1}{|x_v-y|} v^p(y)\mathrm{d}y &\geq C_1 \int_{B_{R_2}(x_v)} \ln \frac{1}{|x_v-y|} v^p(y)\mathrm{d}y\\
 &\geq C_2 N^p \int_0^{R_2} \ln \left(\frac{1}{r}\right)r^{n-1} dr\\
 &\geq \frac{C_2 N^p}{n}R_2^n \ln \left(\frac{1}{R_2}\right)\\
 &\geq \frac{C_3N^{p+1}}{np M^q} \left(\ln\left(\frac{M^q}{N}\right)+\ln p\right)\\
  &\geq \frac{C_4N^{p+1}}{p M^q}\ln\left(\frac{M^q}{N}\right),
\end{split}
\end{equation}
where we have used the inequality
$$
\int_0^\rho r^{n-1} \ln\left(\frac{1}{r}\right) dr \geq \frac{1}{n} \rho^n \ln\left(\frac{1}{\rho}\right)
$$
and the definition of $R_2=\gamma_6 N^\frac{1}{n}M^{-\frac{q}{n}}p^{-\frac{1}{n}}$. Combining \eqref{eq4-10} with \eqref{eq4-11} gives
\begin{equation}\label{eq4-12}
\frac{(M+\gamma_1\gamma_2)M^q}{N^{p+1}}\geq \frac{C_4}{p}\ln\left(\frac{M^q}{N}\right).
\end{equation}
By a similar argument as above, we can also derive
\begin{equation}\label{eq4-13}
\frac{(N+\gamma_1\gamma_2)N^p}{M^{q+1}}\geq \frac{C_4}{q}\ln\left(\frac{N^p}{M}\right).
\end{equation}

Second, by \eqref{eq4-12} and using Lemma \ref{le4.4} (or Remark \ref{rem1} if we take $\Gamma(n,\Omega,\kappa):=\gamma_{7}^{4}$), one gets
\begin{equation}\label{eq4-14}
(1+\gamma_1\gamma_2)\frac{M^{q+1}}{N^{p+1}}\geq \frac{C_4}{p}\ln\left(\frac{M^{q\left(1-\frac{1}{\frac{2}{n}p+1}\right)}}{\gamma_3}\right)\geq \frac{C_4}{p}\ln\left(\frac{M^{\frac{2q}{n+2}}}{\gamma_3}\right)\geq \frac{C_{4}q}{(n+2)p}\ln M.
\end{equation}
Similarly, using \eqref{eq4-13} and Lemma \ref{le4.4} (or Remark \ref{rem1} if we take $\Gamma(n,\Omega,\kappa):=\gamma_{7}^{4}$) again, we obtain
$$
(1+\gamma_1\gamma_2)\frac{N^{p+1}}{M^{q+1}} \geq \frac{C_{4}p}{(n+2)q}\ln N,
$$
which together with \eqref{eq4-14} yields
\begin{equation}\label{eq-a}
  \ln M \cdot \ln N \leq \left[\frac{(n+2)(1+\gamma_1\gamma_2)}{C_{4}}\right]^{2}.
\end{equation}
Since $M>\gamma_{3}2^{\frac{n^{2}K}{2}}\geq2^{\frac{n^{2}}{2}}$ and $N>2^{n}$, we have
$$
M\leq C_5(n,\Omega,\kappa) \quad \text{and} \quad N\leq C_5(n,\Omega,\kappa).
$$
This completes our proof of Lemma \ref{le4.5}.
\end{proof}

\begin{proof}
[Proof of Theorem \ref{th1.3}]
Suppose $n\geq3$ and $pq-1\geq\kappa$ with $\kappa>0$. Assume further that $p\leq Kq$ and $M>\gamma_{3}(n,\Omega,\kappa)2^{\frac{n^{2}K}{2}}$ with $K\geq 1$, then Lemma \ref{le4.5} yields immediately that $M\leq\gamma_{8}(n,\Omega,\kappa)$. In summary, under the assumptions $n\geq3$, $pq-1\geq\kappa$ and $p\leq Kq$, there exists a uniform constant
\[C:=\max\left\{\gamma_3(n,\Omega,\kappa)2^{\frac{n^{2}K}{2}},\gamma_{8}(n,\Omega,\kappa)\right\}\]
depending only on $n$, $\Omega$, $\kappa$ and $K$ such that
\begin{equation}\label{final}
  M=\|u\|_{L^{\infty}(\overline{\Omega})}\leq C.
\end{equation}
Under the assumptions $n\geq3$, $pq-1\geq\kappa$ and $q\leq Kp$, we can also show in the entirely similar way that there exists a uniform constant $C(n,\Omega,\kappa,K)$ such that $N=\|v\|_{L^{\infty}(\overline{\Omega})}\leq C$.

\medskip

Furthermore, if $n\geq3$, $pq-1\geq\kappa$ and $\frac{1}{K}q\leq p\leq Kq$, we deduce that, there exists a uniform constant $C$ depending only on $n$, $\Omega$, $\kappa$ and $K$ such that
\begin{equation*}
\|u\|_{L^{\infty}(\overline{\Omega})}\leq C \quad \mbox{and} \quad \|v\|_{L^{\infty}(\overline{\Omega})}\leq C.
\end{equation*}
This concludes our proof of Theorem \ref{th1.3}.
\end{proof}	

\medskip

From Remark \ref{rem2} and the proof of Lemma \ref{le4.5}, we can also deduce the following Theorem.
\begin{theorem}\label{thm-a}
Let $n\geq3$ and $(u,v)$ be a pair of positive classical solution to \eqref{eq1-1} with $pq-1\geq\kappa$. Assume further that
\begin{equation}\label{assumption-a+}
p\leq K^{q} \quad \text{or} \quad q\leq K^{p}
\end{equation}
for some $K\geq2^{n}$. Then there exists a positive constant $C$ depending only on $n$, $\Omega$, $\kappa$ and $K$ such that
\begin{equation*}
\|u\|_{L^{\infty}(\overline{\Omega})}\leq C \quad \mbox{or} \quad \|v\|_{L^{\infty}(\overline{\Omega})}\leq C.
\end{equation*}
Moreover, if $n\geq4$ is even, we have
\begin{equation*}
  \int_{\Omega}u^{q+1}(x)\mathrm{d}x=\int_{\Omega}v^{p+1}(x)\mathrm{d}x\leq C(n,\Omega,\kappa,K).
\end{equation*}
\end{theorem}
\begin{proof}
Without loss of generalities, we may assume $p\leq K^{q}$. If $M\geq\gamma_{3}K^{n-2}$, it follows from Remark \ref{rem2} and the proof of Lemma \ref{le4.5} that the estimate \eqref{eq-a} still hold, and hence
\begin{equation}\label{eq-a+}
  \ln N \leq \frac{1}{n(n-2)\ln 2}\left[\frac{(n+2)(1+\gamma_1\gamma_2)}{C_{4}}\right]^{2}=:\ln\gamma_{9}(n,\Omega,\kappa),
\end{equation}
Take $\gamma_{10}(n,\Omega,\kappa,K):=\max\left\{\gamma_{3}K^{n-2},\gamma_{9}(n,\Omega,\kappa)\right\}$, then one has $M\leq\gamma_{10}$ or $N\leq\gamma_{10}$. Furthermore, if $n\geq4$ is even, we infer from the system \eqref{eq1-1}, Lemma \ref{le3.1} and integrating by parts that
\begin{equation*}
  \int_{\Omega}u^{q+1}(x)\mathrm{d}x=\int_{\Omega}v^{p+1}(x)\mathrm{d}x\leq \gamma_{1}(n,\Omega,\kappa)\gamma_{10}(n,\Omega,\kappa,K)=:\gamma_{11}(n,\Omega,\kappa,K).
\end{equation*}
This completes our proof of Theorem \ref{thm-a}.
\end{proof}

\end{document}